\pdfoutput=1

\documentclass{amsart} %\documentclass[10pt, a4paper]{article}          % Default-Zeilenabstand

\usepackage[utf8x]{inputenc}
\usepackage{amsmath}
\usepackage{amssymb}
\usepackage{amsthm}
\usepackage{amsfonts}
\usepackage{bm}
\usepackage{graphicx}        % standard LaTeX graphics tool
\usepackage{color}
\usepackage{todonotes}
\usepackage{overpic}
\usepackage[linesnumbered,inoutnumbered]{algorithm2e}
\usepackage{colonequals}
\usepackage[normalem]{ulem}
\usepackage{hyperref}
\usepackage{tikz}
\usetikzlibrary{arrows,shapes,positioning}
\usetikzlibrary{decorations.markings}

\newtheorem{theorem}{Theorem}[section]
\newtheorem{lemma}[theorem]{Lemma}
\newtheorem{corollary}[theorem]{Corollary}
\newtheorem{definition}{Definition}[section]
\newtheorem{remark}{Remark}[section]
\theoremstyle{definition}
\newtheorem{example}{Example}[section]

\newcommand{\R}{\mathbb{R}}
\newcommand{\N}{\mathbb{N}}
\newcommand{\Rinfty}{\mathbb{R}\cup\{\infty\}}
\newcommand{\st}{\;|\;}
\newcommand{\bigst}{\,\big|\,}

\newcommand{\op}[1]{\operatorname{#1}}
\newcommand{\Ecal}{\mathcal{E}}

\newcommand{\Vcal}{\mathcal{V}}

\newcommand{\Acal}{\mathcal{A}}
\newcommand{\Jcal}{\mathcal{J}}
\newcommand{\Mcal}{\mathcal{M}}
\newcommand{\Ccal}{\mathcal{C}}

\newcommand{\Pcal}{\mathcal{P}}
\newcommand{\Scal}{\mathcal{S}}
\newcommand{\Bcal}{\mathcal{B}}
\newcommand{\Id}{\text{Id}}

\newcommand{\ContinuousArgmin}{\tilde{\mathcal{M}}}
\newcommand{\InexactArgmin}{\mathcal{M}}
\newcommand{\ExactArgmin}{\mathcal{M}^{\text{ex}}}

\DeclareMathOperator*{\argmin}{arg\,min}
\providecommand{\inner}[2]{\langle #1,#2 \rangle}

\graphicspath{{gfx/}}

\begin{document}

\title[TNNMG Methods for Block-Separable Minimization Problems]{
Truncated Nonsmooth Newton Multigrid Methods for Block-Separable Minimization Problems}

\author[Gr\"aser]{Carsten Gr\"aser}
\address{Carsten Gr\"aser\\
Freie Universit\"at Berlin\\
Institut f\"ur Mathe\-matik\\
Ar\-nim\-allee~6\\
14195~Berlin\\
Germany}
\email{graeser@mi.fu-berlin.de}

\author[Sander]{Oliver Sander}
\address{Oliver Sander\\
Technische Universität Dresden\\
Institut für Numerische Mathematik\\
Zellescher Weg 12--14\\
01069 Dresden\\
Germany}
\email{oliver.sander@tu-dresden.de}

\dedicatory{In memory of Elias Pipping (1986--2017)}

\thanks{Parts of this work were motivated by enjoyable discussions with Elias Pipping.}

\begin{abstract}
 The Truncated Nonsmooth Newton Multigrid (TNNMG) method is a robust and efficient solution method
 for a wide range of block-separable convex minimization problems, typically stemming from discretizations
 of nonlinear and nonsmooth partial differential equations.  This paper proves global convergence
 of the method under weak conditions both on the objective functional, and on the local inexact
 subproblem solvers that are part of the method.  It also discusses a range of algorithmic choices
 that allows to customize the algorithm for many specific problems.
 Numerical examples are deliberately omitted, because many such examples
 have already been published elsewhere.
\end{abstract}

\maketitle



%%%%%%%%%%%%%%%%%%%%%%%%%%%%%%%%%%%%%%%%%%%%%%%%%%%%%%%%%%%%%%%%%%%%%%%%%%%%%%%%%%%%%%%%%%%%%%%%%%%%%%%%%%

\section{Introduction}

We consider minimization problems for energy functionals having
block-separable nonsmooth terms.
Given an objective functional $\Jcal : \R^n \to \R \cup \{\infty\}$,
we assume that it has the form
\begin{align}
\label{eq:smooth_nonsmooth_split}
 \Jcal = \Jcal_0 + \varphi,
\end{align}
where $\Jcal_0:\R^n \to \R$ is coercive and continuously differentiable,
and $\varphi : \R^n \to \Rinfty$ is block separable, i.e.,
there are functionals
$\varphi_i : \R^{n_i} \to \Rinfty$, $i=1,\dots,M$
that are convex, proper, lower semi-continuous,
and continuous on their domains such that
\begin{align*}
    \varphi(v) = \sum_{i=1}^M \varphi_i(v_i),
\end{align*}
where $\sum_{i=1}^M n_i = n$ and we implicitly identify
$\R^n$ with $\prod_{i=1}^M \R^{n_i}$.

Such minimization problems occur in many fields of science; however, we are
mainly interested in the case where they result from discretized nonsmooth
partial differential equations (PDEs).  Examples of such PDEs include contact and
friction models in solid mechanics~\cite{krause_sander:2005,pipping_sander_kornhuber:2015},
certain models of porous media flow~\cite{berninger_kornhuber_sander:2011}, but
also many variants of Allen--Cahn-type phase-field models~\cite{KornhuberKrause2006}.  The
block-structure is typically either induced by grouping all vector components
of a given vector-valued degree of freedom, or by grouping the degrees of
freedom of a grid element in a Discontinuous-Galerkin-type discretization.
Nonsmoothness can originate from the PDE itself, or from local convex constraints
to admissible sets $K_i$, which can be incorporated through local indicator
functionals $\varphi_i = \chi_{K_i}$.
We do not make use of the relationship of $\Jcal$ to PDEs other than by implicitly assuming that the
individual block sizes $n_i$ are small compared to their number $M$,
and by proposing to solve linear subproblems inexactly by
multigrid methods.

The nondifferentiable terms $\varphi_i$ make
minimizing functionals of the type~\eqref{eq:smooth_nonsmooth_split}
notoriously difficult.
Interior-point methods
\cite{boyd_vandenberghe:2004}
are expensive, because they solve entire sequences of
linear problems, and preserving robust global convergence is difficult,
if these problems are solved only inexactly.
The same issues also plague semismooth Newton~\cite{penot:2012}
and classical active-set methods,
which are closely related
\cite{BergouniouxItoKunisch:1999, Hoppe:1987, hintermueller_ito_kunisch:2002}.
Such methods do not actively exploit convexity and block-separability of the objective functional.
They also struggle with the unbounded gradients that appear in several practically relevant
phase-field models, which lead to ill-conditioned linear systems.

Similarly, the very generic subgradient and bundle methods make no use of the
block structure or convexity, either.  They also disregard any second-order information,
and they are therefore very slow, and convergence is not always guaranteed~\cite{geiger_kanzow:2002}.
In contrast, local relaxation and operator splitting approaches use the
specific problem structure to obtain global convergence, but are also
very slow for problems resulting from discretized PDEs~\cite{carstensen:1997,KornhuberKrause2006}.

In this paper we present the Truncated Nonsmooth Newton Multigrid (TNNMG) method,
a nonsmooth multigrid method for block-separable nonsmooth minimization problems.
While we have proposed TNNMG before for various more specialized problems
\cite{GraeserKornhuber2009,GraeserSander2014,GraeserSackSander2009,graeser_sander:2014},
we now present it as a general framework that includes all previous incarnations
of TNNMG as special cases.

The TNNMG method combines local relaxation methods
with a generalized Newton approach.
It consists of a nonlinear pre-smoother followed by an inexact linear
correction step, typically one multigrid iteration. Thus it can be interpreted
alternatively as a multigrid method with a nonlinear fine grid smoother, or as
an inexact Newton method with a nonlinear corrector step~\cite{sander:2017}.
The method does not regularize the problem, and does not involve parameters that would
need to be selected manually.
Numerous numerical experiments have shown that it achieves multigrid-like convergence
behavior (i.e., mesh-independent convergence rates and linear time complexity)
on a wide range of difficult nonlinear problems if reasonable initial iterates
are available. In case of discretized PDEs, such initial iterates can typically be obtained
by \emph{nested iteration}, i.e., by using inexact solutions for coarser
discretizations \cite{GraeserKornhuber2009}.
We emphasize that good initial iterates are not mandatory for convergence itself,
but only to obtain multigrid-like convergence speed.

In this paper we prove global convergence of the method under very weak conditions
on the functional, the smoother, and the linear correction. The convergence result
subsumes previous results for obstacle~\cite{GraeserKornhuber2009} and contact problems~\cite{GraeserSackSander2009},
as well as for problems with polyhedral nonsmoothness~\cite{GraeserSander2014,KornhuberKrause2006}.
Unlike those previous results, we allow search spaces that do not coincide with the block
structure, and we even allow non-direct sums of search spaces. This in particular
allows us to reprove the main results of~\cite{GraeserSander2014} in a much
simpler way.
We show convergence of the presented method to stationary
points, such that we get convergence to minimizers, e.g.,
if $\Jcal_0$ is convex.
While the convergence proof only requires continuous
differentiability of $\Jcal_0$, the method uses a
Newton-type linearization in a substep such that
we need at least Lipschitz continuity of $\Jcal_0'$.
The strongest results are obtained for strictly convex $\Jcal$, but biconvex and certain quasiconvex
problems are covered by the theory as well.

The convergence proof relies mainly on the convergence of the Gauß--Seidel-type pre-smoother
by itself. Weak conditions on such smoothers are stated that ensure global convergence
of TNNMG. These weak conditions allow various cheap inexact solvers to be used for the
local nonsmooth subproblems. We propose various such inexact solvers and prove that
they fulfill the conditions.

On the other hand, while the linear correction step does not really interfere in the
convergence proof at all, it is crucial for \emph{fast} convergence of the TNNMG method.
The step involves constructing a sufficiently large subspace on which a Newton-type correction
problem is well defined, and then solving this problem in a suitably inexact way.
The standard way is to use a single geometric or algebraic multigrid step here.
Such a choice then leads to the interpretation of the overall method as a multigrid
method. However, other choices of solvers are possible, and can be useful in
certain situations.

\bigskip

This article is structured as follows:  We begin by formally stating the
problem in Chapter~\ref{sec:block_separable_functionals}.
In Chapter~\ref{sec:tnnmg} we introduce the TNNMG method,
and in Chapter~\ref{sec:convergence} we prove its global convergence.
Chapter~\ref{sec:smoothers} discusses ways to inexactly solve the local
minimization problems that make up the nonlinear multigrid smoother.
Chapter~\ref{sec:coarse_grid_corrections} introduces truncated
linearized correction problems, and Chapter~\ref{sec:linear_mg} discusses efficient inexact
solvers for those problems, which are crucial
to obtain the overall multigrid-like convergence speed of TNNMG.
An appendix collects a few technical results.

The paper deliberately omits numerical benchmarks of the method, because many of these have already
appeared in the literature.  Among others, we refer the reader to \cite{GraeserSackSander2009}
for contact problems, \cite{GraeserSackSander2009,graeser_sander:2014} for phase-field problems,
and \cite{sander:2017} for the performance of TNNMG on small-strain plasticity problems.

\section{Block-separable nonsmooth functionals}
\label{sec:block_separable_functionals}

The TNNMG method is designed to solve nonsmooth minimization problems
\begin{align}
\label{eq:min_problem_revisited}
    u^* \in \R^n: \qquad
        \Jcal(u^*) \leq \Jcal(v) \qquad \forall v \in \R^n,
\end{align}
for possibly nondifferentiable functionals $\Jcal:\R^n \to \Rinfty$.
Throughout the paper we will assume that $\Jcal$ is
proper, coercive, lower semi-continuous, continuous on its domain
$\op{dom} \Jcal \colonequals \{x \in \R^n \st \Jcal(x) < \infty\}$,
and that $\op{dom} \Jcal$ is convex, but not necessarily closed.

The algorithm is based on a subspace decomposition
\begin{align}\label{eq:subspace_splitting}
    \R^n = \sum_{k=1}^m V_k
\end{align}
into $m$ subspaces $V_k$ of dimensions $n_k \colonequals \dim V_k$, $k=1,\dots,m$.
The sum in \eqref{eq:subspace_splitting} does not need to be direct.
For practical purposes we will also assume coordinate systems for each
subspace in the sense that we identify $V_k$ with $\R^{n_k}$ using
prolongation maps
\begin{align}
 \label{eq:prolongation_operators}
    P_k : \R^{n_k} \to \R^n,
    \qquad
    k = 1,\dots,m,
\end{align}
such that for each $k$, $P_k$ is an isomorphism from $\R^{n_k}$ to $V_k$.

If $\Jcal$ is differentiable,
then the splitting \eqref{eq:subspace_splitting}
is sufficient to characterize minimizers of $\Jcal$ in the sense that
$u \in \R^n$ is a global minimizer of $\mathcal{J}$ if it is a global minimizer
with respect to each $V_k$.
For nondifferentiable $\Jcal$ this is no longer true,
unless the subspace splitting is compatible with the nondifferentiable
structure of $\Jcal$. For counterexamples we refer to \cite{GlowinskiLionsTremolieres81}.
Here, we will make compatibility an additional assumption.
\begin{definition}
    The decomposition \eqref{eq:subspace_splitting} is called
    \emph{compatible with $\Jcal$} if
    \begin{align}\label{eq:subspace_optimality}
        \Jcal(u) \leq \Jcal(u + v) \qquad \forall v \in V_k, \quad k=1,\dots,m
    \end{align}
    implies global optimality
    \begin{align*}
        \Jcal(u) \leq \Jcal(u+v) \qquad \forall v \in \R^n.
    \end{align*}
\end{definition}

An important example for this situation are block-separable nonsmooth
problems with a corresponding subspace decomposition.
We consider problems with a nonsmooth, convex, and separable
part, and a smooth but possibly nonconvex part.

\begin{definition}
    \label{def:block_separable}
    Let $R : \R^n \to \prod_{k=1}^M \R^{N_k}$ be an isomorphism with
    $R = (R_1, \dots, R_M)$ and surjective linear maps $R_k : \R^n \to \R^{N_k}$.
    We say that a function $\Jcal: \R^n \to \R \cup \{\infty\}$
    is \emph{block-separable nonsmooth with respect to $R$}
    if there is a continuously differentiable function $\Jcal_0: \R^n \to \R$,
    and convex, proper, lower-semicontinuous functions
    $\varphi_k:\R^{N_k} \to \R \cup \{ \infty \}$
    such that
    \begin{align*}
        \Jcal(v) = \Jcal_0(v) + \underbrace{\sum_{k=1}^M \varphi_k (R_k v)}_{\equalscolon \varphi(v)}.
    \end{align*}
\end{definition}

We note that for any such isomorphism $R : \R^n \to \prod_{k=1}^M \R^{N_k}$,
the maps $R_k : (\ker R_k)^\perp \to \R^{N_k}$ are isomorphisms themselves,
which together with $n = \sum_{k=1}^M N_k$ gives rise to the direct
sum representation
\begin{align}\label{eq:direct_sum}
    \R^n = \bigoplus_{k=1}^M (\op{ker}R_k)^\perp.
\end{align}

Using the subspaces induced by the block-separability leads to a compatible decomposition.
However, the following result shows that more general decompositions are also possible.

\begin{lemma}
    \label{lem:block_separable_compatible}
    Let $\Jcal: \R^n \to \Rinfty$ be convex and block-separable nonsmooth
    with respect to an isomorphism $R : \R^n \to \prod_{k=1}^M \R^{N_k}$.
    Assume that for any block $k=1,\dots,M$ the decomposition contains
    a subspace $V_{k'}$ with $(\op{ker} R_k)^\perp \subset V_{k'}$.
    Then the decomposition is compatible with $\Jcal$.
\end{lemma}

\begin{proof}
    Let $u \in \R^n$ satisfy \eqref{eq:subspace_optimality}.
    To any $k = \{1,\dots,M\}$ we associate an index $k'$ such that $(\op{ker} R_k)^\perp \in V_{k'}$.
    The optimality of $u$
    in $u+V_{k'}$ implies the variational inequalities
    \begin{align}\label{eq:subspace_vi}
        \inner{\Jcal_0'(u)}{w_{k'}-u} + \varphi(w_{k'}) - \varphi(u) \geq 0
        \qquad \forall w_{k'} \in u+V_{k'}
    \end{align}
    for $k=1,\dots,M$.
    Now let $w \in \R^n$ and $v=w-u$.
    Then, according to \eqref{eq:direct_sum},
    there is a decomposition
    \begin{align*}
        v &= \sum_{k=1}^M v_k,
        \qquad
        v_k \in (\op{ker}R_k)^\perp \subset V_{k'}
    \end{align*}
    with some $k'$ for each $k$.
    Since the sum  \eqref{eq:direct_sum} is direct, we have
    $R_k v = R_k v_k$ and $R_i v_k = 0$ for $i\neq k$, which implies
    \begin{align}\label{eq:nonlinearity_splitting_auxiliary}
        \varphi(u+v_k) - \varphi(u)
        = \sum_{i=1}^M \big[ \varphi_i(R_i (u+ v_k)) - \varphi_i(R_i u) \big]
        = \varphi_k(R_k (w)) - \varphi_k(R_k u).
    \end{align}

    Inserting $w_{k'}=u+v_k$ into the variational inequalities \eqref{eq:subspace_vi},
    summing up for all $k = 1,\dots,M$, and using \eqref{eq:nonlinearity_splitting_auxiliary}
    finally gives
    \begin{align*}
        \inner{\Jcal_0'(u)}{w-u} + \varphi(w) - \varphi(u) \geq 0.
    \end{align*}
    Since $w \in \R^n$ was arbitrary and $\Jcal$ is convex,
    this proves the assertion.
\end{proof}

Lemma~\ref{lem:block_separable_compatible} shows compatibility for general
possibly non-direct subspace decompositions. Such decompositions occur, e.g.,
when using domain decomposition methods
\cite{HeimsundTaiXu2002, Tai2003, TaiXu2001,carstensen:1997}
and nonlinear multi-level
relaxation methods \cite{Mandel1984, GraeserKornhuber2009, Badea2006, Kornhuber1994, Kornhuber1996}.
In the present paper we will concentrate on
subspace decompositions where we either have
\begin{align}
    \label{eq:decomposition_by_blocks}
    (\op{ker} R_k)^\perp = V_k \qquad k=1,\dots,m=M
\end{align}
or where each $(\op{ker} R_k)^\perp$ can be written as a direct or non-direct sum of a subset
of the subspaces $V_{k'}$.
In the following we list a few examples of suitable
energy functionals. For a set $U$, we denote by $\chi_U$ the indicator functional
\begin{align*}
    \chi_U(z) \colonequals
    \begin{cases}
        0 & \text{if $z \in U$},\\
        \infty & \text{otherwise}.
    \end{cases}
\end{align*}

\begin{example}
    Let $A \in \R^{n \times n}$ be symmetric and positive definite, $b \in \R^n$,
    and let $K_k \subset \R$, $k=1,\dots,n$ be nonempty, closed, possibly unbounded intervals.
    Then
    \begin{equation*}
        \Jcal(v)
        \colonequals
        \underbrace{\frac{1}{2} \langle Av,v\rangle - \langle b,v\rangle}_{\equalscolon \Jcal_0(v)}
        + \sum_{k=1}^n \underbrace{\chi_{K_k}(v_k)}_{\equalscolon \varphi_k(v_k)}
    \end{equation*}
    is block-separable nonsmooth, and the decomposition of $\R^n$ into
    $V_k = \{v \in \R^n \st v_i = 0 \quad \forall i\neq k\}$
    is compatible.
    This is the energy functional for the classic obstacle problem of
    minimizing $\Jcal_0$ in the hypercube $K=\prod_{i=1}^n K_i$, see~\cite{GraeserKornhuber2009}.
\end{example}

\begin{example}
    \label{example:anisotropic_problem}
    Let $D_i \in \R^{d \times n}$, and $\gamma_i : \R^{d} \to \R$ be convex,
    continuously differentiable functions for $i=1,\dots,l$. Then the functional
    \begin{align*}
        \Jcal(v) \colonequals \underbrace{\sum_{i=1}^{l} \gamma_i (D_iv)}_{\equalscolon \Jcal_0(v)}
        + \sum_{k=1}^n \underbrace{\chi_{[0,1]}(v_k)}_{\equalscolon \varphi_k(v_k)}
    \end{align*}
    is block-separable nonsmooth, and the decomposition into
    $V_k = \{v \in \R^n \st v_i = 0 \quad \forall i\neq k\}$
    is compatible.
    Functionals of this form with $\gamma_i(z) = \omega_i\gamma(z)$ and a quadrature
    weight $\omega_i$ are obtained, e.g., for discretized minimal surface equations
    with an obstacle \cite{Hardering2013}, and shallow-ice glacier models \cite{BuelerEtAl2012, GraeserJouvet2012}.
    In both cases $D_i$ represents the local evaluation of the gradient of a finite
    element function. For the minimal surface equation we have
    $\gamma(z) = \sqrt{1+\|z\|_2^2}$, while the $p$-Laplace operator in shallow-ice
    models leads to $\gamma(z) = \|z\|_2^p$.
    For anisotropic phase-field models one obtains similar functions
    with an additional quadratic term~\cite{GraeserKornhuberSack2011}.
\end{example}

\begin{example}
\label{ex:vector_valued_phasefield}
    Let $A$ and $b$ be as above,
    $R: \R^n \to (\R^L)^m$ an isomorphism, %inducing a block-structure
    and
    \begin{align}
    \label{eq:gibbs_simplex}
        G \colonequals \Big\{v \in \R^L \st \sum_{i=1}^L v_i=1, \; v_i\geq 0 \quad \forall i =1,\dots,L\Big\}
    \end{align}
    the $L$-dimensional Gibbs-simplex.
    Then the functional
    \begin{align*}
        \Jcal(v) \colonequals \underbrace{\inner{Av}{v} - \inner{b}{v}}_{\equalscolon \Jcal_0(v)}
        + \sum_{k=1}^m \underbrace{\chi_{G}(R_k v)}_{\equalscolon \varphi_k(R_k v)}
    \end{align*}
    is block-separable nonsmooth, and the decomposition into
    $V_k = (\op{ker} R_k)^\perp$
    is compatible. Note that these subspaces take the form
    $V_k = e_k \otimes \R^L$ if we identify $\R^n = (\R^L)^m$, where $e_k$ is the $k$-th canonical basis vector
    in $\R^m$.
    Functionals of this form occur when discretizing
    multi-component phase-field models~\cite{KornhuberKrause2006}.
    It was shown in \cite{GraeserSander2014,KornhuberKrause2006} that another compatible subspace
    decomposition is given by $V_{k,i} = e_k \otimes \eta_i$,
    where $\eta_1, \dots,\eta_l$ denote the $l=\frac12(L-1)L$
    edge vectors of $G$.
\end{example}

\begin{example}
\label{ex:norm_nondifferentiability}
    Let $A$, $b$, and $R$ as above, $\|\cdot\| : \R^L \to \R$ a norm on $\R^L$,
    and $\omega_k >0$ positive weights.
    Then the functional
    \begin{align*}
        \Jcal(v) = \underbrace{\inner{Av}{v} - \inner{b}{v}}_{\equalscolon\Jcal_0(v)}
        + \sum_{k=1}^m \underbrace{\omega_k\|R_k v\|}_{\equalscolon\varphi_k(R_k v)}
    \end{align*}
    has the desired form with a compatible decomposition
    of $\R^n = (\R^L)^m$ into $V_k = e_k \otimes \R^L$.
    Functionals of this type are, e.g., obtained for certain friction laws,
    and in primal formulations of small-strain plasticity~\cite{han_reddy:2013}.
\end{example}

\section{The Truncated Nonsmooth Newton Multigrid method}
\label{sec:tnnmg}

The \emph{Truncated Nonsmooth Newton Multigrid Method} (TNNMG) was introduced
in \cite{GraeserKornhuber2009} for quadratic obstacle problems, and later generalized to variational
inequalities of the second kind with separable nonsmooth nonlinearities \cite{GraeserSackSander2009, Graeser2011}.
Similar to these cases, the extension to block-separable nonsmooth problems of the form~\eqref{eq:min_problem_revisited}
will be based on a nonlinear block Gau\ss--Seidel iteration and
an additional linear correction step.
For given initial iterate $u^0 \in \op{dom} \Jcal$ and iteration number $\nu \in \N_0$,
one step of the TNNMG method is defined as follows:

%Let $\nu \in \N_0$ be the iteration number and $u^0$ a given initial iterate.
%One step of the TNNMG method for \eqref{eq:min_problem_revisited} with a nonlinear
%block-Gauß--Seidel smoother is defined as follows:

\medskip

\begin{algorithm}[H]
\SetAlgoLined
 \KwIn{Given $u^\nu$}\label{alg:tnnmg_smoothing}
    \Begin(Nonlinear pre-smoothing){
        Set $w^{\nu,0} \colonequals u^\nu$\\
  \For{$k=1,\dots,m$}{
      Compute $w^{\nu,k} \in w^{\nu,k-1} + V_k$: \label{enum:gs_local_min}
   \begin{align*}
       w^{\nu,k} \colonapprox
     \argmin_{v \in w^{\nu,k-1} + V_k} \Jcal(v)
%     = w^{k-1} + P_k \argmin_{\xi \in \R^{n_k}} \tilde{\Jcal}_{w^{k-1},k}(\xi)
   \end{align*}
  }
    Set $u^{\nu+\frac12} \colonequals w^{\nu,m}$
 } \label{alg:tnnmg_smoothing_final}
\end{algorithm}

% Manually break the algorith.  Without this page break we get looots
% of free space on the page immediately preceding the algorithm.
\newpage

\begin{algorithm}[H]
\SetAlgoLined
\setcounter{AlgoLine}{8}
    \Begin(Truncated linear correction){\label{alg:tnnmg_coarse}
    Determine large subspace $W_\nu \subset \R^n$ such that $\Jcal|_{u^{\nu+\frac12}+W_\nu}$ is $C^2$ near $u^{\nu+\frac12}$.\\
  Compute $v^\nu \in W_\nu$ as \label{enum:coarse_correction}
  \begin{align}
  \label{eq:inexact_newton_problem}
      v^\nu \colonapprox
          -\Bigl(\Jcal''(u^{\nu+\frac12})|_{W_\nu \times W_\nu}\Bigr)^{-1}
          \Bigl(\Jcal'(u^{\nu+\frac12})|_{W_\nu} \Bigr)
  \end{align}
 }
 \Begin(Post-processing){
  Compute the projection $\tilde{v}^\nu \colonequals \Pi_{\op{dom}\Jcal - u^{\nu+1/2}}(v^\nu)$ \label{enum:tnnmg_projection} \\
  Compute $\rho_\nu \in [0,\infty)$ such that $\Jcal(u^{\nu+\frac12} + \rho_\nu \tilde{v}^\nu) \leq \Jcal(u^{\nu+\frac12})$
 }
 \KwOut{Set $u^{\nu+1} \colonequals u^{\nu+\frac12} + \rho_\nu \tilde{v}^\nu$}
  \label{alg:tnnmg_final}
\end{algorithm}

\medskip

Using the prolongation operators $P_k$ of~\eqref{eq:prolongation_operators},
the incremental minimization problems in Line~\ref{enum:gs_local_min}
for a given intermediate iterate $w = w^{\nu,k-1} \in \R^n$ can be written as minimization
problems in $\R^{n_k}$
\begin{align}\label{eq:TNNMG_block_subproblem}
 \argmin_{v \in w + V_k} \Jcal(v)
 = w + \argmin_{v \in V_k} \Jcal(w+v)
 = w + P_k \argmin_{\xi \in \R^{n_k}} \Jcal(w+ P_k \xi).
\end{align}
We are mainly interested in block-separable problems
originating from discretized PDEs where the block-sizes $n_k$ are
independent of the problem size $n$.
If the subspaces $V_k$ are chosen according to the blocks or even finer,
then the size of the subproblems \eqref{eq:TNNMG_block_subproblem}
is in $O(1)$, which allows to use more expensive schemes
for their approximate solution.

In Line~\ref{enum:coarse_correction} we denote by
$\Jcal'(u^{\nu+\frac12})|_{W_\nu}$
and $\Jcal''(u^{\nu+\frac12})|_{W_\nu \times W_\nu}$ the gradient and Hessian of the restriction of the
objective functional $\Jcal$ to the subspaces $W_\nu$ and
$W_\nu \times W_\nu$, respectively.
While $\Jcal$ is in general not even differentiable, these are still well-defined expressions
since $\Jcal$ is twice continuously differentiable on $W_\nu$ by construction of $W_\nu$.
In implementations it is convenient to represent these restricted derivatives with respect
to coordinates in $\R^n$ by extending them with zero
to the orthogonal complement of $W_\nu$.
In such a situation, the restricted Hessian is not invertible as a map $\R^n \to \R^n$.
However, if $\Jcal$ is convex and $\Jcal_0''(v)$ exists and is positive definite
for any $v$ then the correction problem~\eqref{eq:inexact_newton_problem} still has a unique solution in $W_\nu$.
In general this can be expressed by replacing
$(\Jcal''(u^{\nu+\frac12})|_{W_\nu \times W_\nu})^{-1}$ by
the Moore--Penrose pseudo-inverse of $\Jcal''(u^{\nu+\frac12})$ on $\R^n$.
Similar arguments hold if $\Jcal_0$ is not $C^2$, but has a Lipschitz-continuous first
derivative. Then, $\Jcal''$ needs to be replaced by a suitable generalized derivative of $\Jcal'$.

Simply adding the linear coarse grid correction
$v^\nu$ from Line~\ref{alg:tnnmg_coarse} to $u^{\nu + \frac{1}{2}}$ to obtain
the new iterate $u^{\nu+1}$ may lead to
infeasibility of $u^{\nu+1}$, since Line~\ref{alg:tnnmg_coarse}
is not aware of the domain of $\Jcal$.
Adding a damped version of $v^\nu$ can ensure feasibility, but
this may lead to very small damping parameters
and thus to poor convergence.
As an alternative we first compute a Euclidean projection $\tilde{v}^\nu$ of $v^\nu$ into
the domain in Line~\ref{enum:tnnmg_projection}.
While this ensures feasibility of $u^{\nu+\frac12} + \tilde{v}^\nu$
it may still increase energy, because the projection is
only aware of the domain, but not of the values of $\Jcal$.
To ensure energy decrease damping is then applied to the
projected correction.

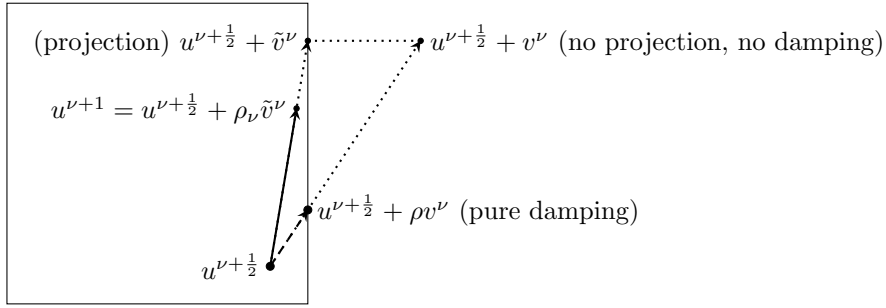
\begin{figure}
    \begin{tikzpicture}[ scale=1.0 ]

    \draw[-] (0,0) -- (4,0) -- (4,4) -- (0,4) -- cycle;

    \coordinate (current) at (3.5,0.5);
    \coordinate (damped) at (4,1.25);
    \coordinate (undamped) at (5.5,3.5);
    \coordinate (projected) at (4,3.5);
    \coordinate (dampedprojected) at (3.85,2.6);

    \filldraw (current) circle (1.5pt);
    \draw (current) node[left] {$u^{\nu+\frac12}$};

    \filldraw (damped) circle (1.5pt);
    \draw (damped) node[right] {$u^{\nu+\frac12} + \rho v^\nu$ (pure damping)};

    \filldraw (undamped) circle (1pt);
    \draw (undamped) node[right] {$u^{\nu+\frac12} + v^\nu$ (no projection, no damping)};

    \filldraw (projected) circle (1pt);
    \draw (projected) node[left] {(projection) $u^{\nu+\frac12} + \tilde{v}^\nu$};

    \filldraw (dampedprojected) circle (1pt);
    \draw (dampedprojected) node[left] {$u^{\nu+1}=u^{\nu+\frac12} + \rho_\nu\tilde{v}^\nu$};

    \draw (current)  edge[dotted, arrows=->, >=stealth, thick] (undamped);
    \draw (current)  edge[dashed, arrows=->, >=stealth, thick] (damped);
    \draw (undamped) edge[dotted, thick] (projected);
    \draw (current)  edge[dotted, arrows=->, >=stealth, thick] (projected);
    \draw (current)  edge[arrows=->, >=stealth, thick] (dampedprojected);
    \end{tikzpicture}
    \caption{Illustration of the TNNMG post-processing steps.
      The square represents the admissible set $\op{dom} \Jcal$.}
    \label{fig:tnnmg_post_processing}
\end{figure}

The post-processing steps are depicted in Figure~\ref{fig:tnnmg_post_processing}.
In particular, note that using a damped, projected
correction may lead to larger steps compared to pure damping
without preceding projection.

\section{Global convergence}
\label{sec:convergence}

We now prove that the TNNMG iteration converges to a stationary point for any initial iterate $u^0$.
For this we need some assumptions on the inexact solution operators
that map $w^{\nu,k-1}$ to $w^{\nu,k}$ in Line~\ref{enum:gs_local_min} of the TNNMG algorithm.
To this end we introduce operators
\begin{align}
    \label{eq:local_inexact_min_op}
    \InexactArgmin_k:\op{dom}\Jcal \to \op{dom} \Jcal, \qquad
    \InexactArgmin_k - \Id : \op{dom}\Jcal \to V_k
\end{align}
such that
\begin{align}
    \label{eq:local_iterates}
    w^{\nu,k} = \InexactArgmin_k(w^{\nu,k-1}), \qquad w^{\nu,0} = u^\nu.
\end{align}
These operators represent the inexact minimization of $\Jcal$ in the
affine subspaces $w^{\nu,k-1} + V_k$ by, e.g., some iterative algorithm.
Using this notation the inexact block Gau\ss--Seidel loop in
Lines~\ref{alg:tnnmg_smoothing}--\ref{alg:tnnmg_smoothing_final}
takes the form
\begin{align}
    \label{eq:gs_substep}
    u^{\nu+\frac12} = \InexactArgmin(u^\nu), \qquad
    \InexactArgmin \colonequals \InexactArgmin_m \circ \dots \circ \InexactArgmin_1.
\end{align}

The convergence proof hinges on the continuity of the compound operators $\Jcal \circ \InexactArgmin_k$.
In case of exact minimization, $\InexactArgmin_k$ is
\begin{align*}
    \InexactArgmin_k(\cdot) = \argmin_{v \in (\cdot)+V_k} \Jcal(v),
\end{align*}
and we get continuity of
$\Jcal \circ \InexactArgmin_k = \min_{v \in (\cdot)+V_k} \Jcal(v)$ as a direct consequence of
Lemma~\ref{lem:exact_min_continuous} below.
However, this continuity may not hold for general inexact minimization
operators $\InexactArgmin_k$.
Therefore we introduce additional operators
\begin{align}
    \label{eq:local_continuous_min_op}
    \ContinuousArgmin_k:\op{dom}\Jcal \to \op{dom} \Jcal, \qquad
    \ContinuousArgmin_k - \Id : \op{dom}\Jcal \to V_k
\end{align}
with continuous $\Jcal \circ \ContinuousArgmin_k$
%$\Mcal_k:\op{dom}\Jcal \to \op{dom} \Jcal$ with
%$\Mcal_k - \Id : \op{dom}\Jcal \to V_k$
that bound
the energy decrease realized by $\InexactArgmin_k$ from above, i.e., we assume that
\begin{align*}
    \Jcal(\InexactArgmin_k(v)) \leq \Jcal(\ContinuousArgmin_k(v)) \qquad \forall v \in \op{dom}\Jcal.
\end{align*}
These operators do not appear in the implementation of the
method but will only be used as a tool to prove convergence.

%Since this substep represents an inexact minimization in $w^{\nu,k-1}+V_k$
%we require that $\tilde{\Mcal}_k(v) - v \in V_k$ for any $v \in \op{dom}\Jcal$.
Finally, we introduce an operator $\Ccal:\op{dom}\Jcal \to \op{dom}\Jcal$
that represents the truncated linear correction and the post-processing in
Lines~\ref{alg:tnnmg_coarse}--\ref{alg:tnnmg_final} in the sense that
\begin{align}
    \label{eq:coarse_correction_substep}
    u^{\nu+1} = \Ccal(u^{\nu+\frac12}).
\end{align}
However, this operator does not have to be continuous in any way.
Note that~\eqref{eq:gs_substep} and~\eqref{eq:coarse_correction_substep} together
formalize the TNNMG algorithm.

\begin{theorem}
    \label{thm:gs_stationary}
    Let $\Jcal : \R^n \to \Rinfty$ be coercive, proper, lower-semicontinuous,
    and continuous on its domain, and assume that
    $\Jcal(w+(\cdot))$ has a unique global minimizer in all $V_k$, $k=1,\dots,m$
    for each $w\in \op{dom}\Jcal$.
    For local operators $\InexactArgmin_k$ of the form \eqref{eq:local_inexact_min_op},
    a correction operator $\Ccal:\op{dom}\Jcal \to \op{dom}\Jcal$, and an initial guess
    $u^0 \in \op{dom} \Jcal$
    let $(u^\nu)$ be given by the algorithm
    \eqref{eq:gs_substep} and \eqref{eq:coarse_correction_substep}.
    Assume that there are operators $\ContinuousArgmin_k$ of the form
    \eqref{eq:local_continuous_min_op} such that the following holds:
    \begin{enumerate}
        \item
            Monotonicity: $\Jcal(\InexactArgmin_k(v)) \leq \Jcal(\ContinuousArgmin_k(v)) \leq \Jcal(v)$
            and $\Jcal(\Ccal(v)) \leq \Jcal(v)$ for all $v \in \op{dom}\Jcal$.
        \item
            Continuity: $\Jcal \circ \ContinuousArgmin_k$ is continuous.
        \item\label{enum:stability}
            Stability: $\Jcal(\ContinuousArgmin_k(v)) < \Jcal(v)$ if $\Jcal(v)$ is not minimal in $v + V_k$.
    \end{enumerate}
    Then any accumulation point $u$ of $(u^\nu)$ is stationary in the sense that
    \begin{align}\label{eq:stationary_point}
        \Jcal(u) \leq \Jcal(u + v) \qquad \forall v \in V_k, \quad k=1,\dots,m.
    \end{align}
\end{theorem}
The proof uses the following direct consequence of lower semi-continuity.
\begin{lemma}
    \label{lemma:closed_sublevelset}
    If $F: \R^n \to \Rinfty$ is lower semi-continuous,
    then the sub-level set $F^{-1}((-\infty,C])$
    is closed for any $C\in \R$. In particular, this guarantees that $\lim x^\nu \in \op{dom} F$
    for any convergent sequence $x^\nu$ with bounded $F(x^\nu)$.
\end{lemma}
\begin{proof}[Proof of Theorem~\ref{thm:gs_stationary}]
    Throughout the proof we assume that the sequences $w^{\nu,k}$ are defined according to
    \eqref{eq:local_iterates}.
    Now let $(u^{\nu_l})$ be any convergent subsequence of $(u^\nu)$ with $u^{\nu_l} \to u$.
    Then by Lemma~\ref{lemma:closed_sublevelset} we have $u \in \op{dom} \Jcal$.
    We will show that $u$ is stationary in the sense of \eqref{eq:stationary_point}.

    As the subsequences $(w^{\nu_l,k})_{l \in \N}$ are also bounded for any $k$ there are
    subsubsequences (w.l.o.g.\ also indexed by $\nu_l$) and limits $w^k$ such that
    $w^{\nu_l,k} \to w^k$ for $l\to \infty$ for all $k=1,\dots,m$.
    In particular, we
    have $w^{\nu_m,0} = u^{\nu_m} \to u \equalscolon w^0$.
    For any $k=1,\dots,m$ we then have by monotonicity
    of $\Jcal$ with respect to $\InexactArgmin_k$, $\ContinuousArgmin_k$, and $\Ccal$
    \begin{align*}
        \Jcal(w^{\nu_{l+1},k-1})
            \leq \Jcal(w^{\nu_l,k})
            \leq \Jcal(\ContinuousArgmin_{k}(w^{\nu_l,k-1}))
            \leq \Jcal(w^{\nu_l,k-1}) \leq \Jcal(u^0)<\infty,
    \end{align*}
    and thus $w^k \in \op{dom}\Jcal$ by Lemma~\ref{lemma:closed_sublevelset}.
    Taking the limit $l \to \infty$ and using continuity of $\Jcal$ and $\Jcal \circ \ContinuousArgmin_{k}$ we get
    \begin{align}
        \label{eq:intermediate_iterate_no_descent}
        \Jcal(w^{k-1})
        \leq \Jcal(w^k)
            \leq \Jcal(\ContinuousArgmin_{k}(w^{k-1}))
            \leq \Jcal(w^{k-1}),
    \end{align}
    which, together with stability, implies that $\ContinuousArgmin_k(w^{k-1})$ is the unique
    minimizer of $\Jcal$ in $w^{k-1}+V_k$ and thus $\ContinuousArgmin_k(w^{k-1}) = w^{k-1}$.
    From $w^{\nu_l,k}-w^{\nu_l,k-1} \in V_k$ we get for the limit $w^k-w^{k-1} \in V_k$
    and thus $w^k \in w^{k-1} + V_k$.
    Hence $\Jcal(w^k) \leq \Jcal(w^{k-1})$ (as shown in \eqref{eq:intermediate_iterate_no_descent})
    and uniqueness of the minimizer in $w^{k-1} + V_k$ show $w^k=w^{k-1}$.

    We conclude that we have $w^m = w^{m-1} = \dots = w^0=u$, and thus
    $\ContinuousArgmin_k(u) = u$ for $k=1,\dots,m$.
    Hence we get \eqref{eq:stationary_point} from the stability assumption~(\ref{enum:stability}).
\end{proof}

Under the additional assumption of a compatible subspace decomposition
we can show convergence to minimizers of $\Jcal$.

\begin{corollary}
    \label{cor:gs_accumulation_points}
    Assume that, additionally to the assumptions of Theorem~\ref{thm:gs_stationary},
    the subspace decomposition $\R^n = \sum_{k=1}^m V_k$ is compatible with $\Jcal$.
    Then any accumulation point of $u^\nu$ is a global minimizer of $\Jcal$.
\end{corollary}

\begin{corollary}
    \label{cor:gs_convergence}
    Assume that, additionally to the assumptions of Corollary~\ref{cor:gs_accumulation_points},
    $\Jcal$ has a unique global minimizer $u^*$, then $u^\nu$ converges to $u^*$.
\end{corollary}

\begin{proof}
    Since $\Jcal$ is coercive and $\Jcal(u^\nu)$ is monotonically decreasing,
    the sequence $u^\nu$ is bounded.
    Hence, if $u^\nu$ does not converge to $u^*$ it must have an accumulation point
    $u \neq u^*$, which contradicts Corollary~\ref{cor:gs_accumulation_points}.
\end{proof}

These convergence results cover many important cases; in particular, they imply global convergence
to a minimizer for all example functionals listed in Chapter~\ref{sec:block_separable_functionals}.
For reference we state global convergence for strictly convex, block-separable nonsmooth functionals
as a separate corollary.

\begin{corollary}
 Let $\Jcal : \R^n \to \R \cup \{ \infty \}$ be strictly convex, coercive, proper, lower semi-continuous,
 and continuous on its domain.  Assume that $\Jcal$ is block-separable nonsmooth, and consider the
 induced subspaces $V_k \colonequals (\op{ker} R_k)^\perp$.
 Then the TNNMG method with smoother
 $\InexactArgmin = \InexactArgmin_m \circ \dots \circ \InexactArgmin_1$
 and linear correction $\Ccal$ satisfying the \emph{monotonicity}, \emph{continuity}, and \emph{stability}
 assumptions of Theorem~\ref{thm:gs_stationary} converges globally to the unique minimizer of $\Jcal$.
\end{corollary}

This corollary subsumes previous results
from~\cite{GraeserSander2014,KornhuberKrause2006,GraeserSackSander2009} and others.
However, the theory covers many more general situations as well, such as the biconvex energies
in~\cite{miehe_welschinger_hofacker_2010}, and certain quasi-convex functionals.

\section{Solution of local subproblems}
\label{sec:smoothers}

In the following we will give several examples of local operators
$\InexactArgmin_k$ that satisfy the assumptions of Theorem~\ref{thm:gs_stationary}.%
\footnote{Note that ``satisfy the assumptions'' generally means that there exists
a bounding minimization operator $\ContinuousArgmin_k$ with the necessary properties.
In some cases this operator coincides with $\InexactArgmin_k$ itself.}
Evaluating such an operator is equivalent to the inexact minimization
\begin{align}
\label{eq:local_min_problem_values}
 \argmin_{v \in w^{\nu,k-1} + V_k} \Jcal(v)
%     = w^{k-1} + P_k \argmin_{\xi \in \R^{n_k}} \tilde{\Jcal}_{w^{k-1},k}(\xi)
\end{align}
in the $k$-th affine subspace $w^{\nu,k-1} + V_k$.
These evaluations appear as local subproblems
in Line~\ref{enum:gs_local_min} of the TNNMG algorithm,
and form the nonlinear Gauß--Seidel smoother.
In the following we only consider a single subspace $V \colonequals V_k$.
Introducing $N = n_k$, and
dropping all other indices we will write \eqref{eq:local_min_problem_values} as
\begin{align}
    \label{eq:local_subproblem}
 \argmin_{v \in w + V} \Jcal(v)
\end{align}
for a given $w \in \R^n$.

\subsection{Exact minimization}

The trivial choice for the local correction operator $\InexactArgmin$
is the exact minimization operator $\ExactArgmin$ given by
\begin{align}
\label{eq:exact_minimization}
    \ExactArgmin(w) \colonequals \argmin_{v \in w +V} \Jcal(v).
\end{align}
In the case where the minimizer is not unique, we assume that
$\ExactArgmin(w)$ is any of the global minimizers in the affine
subspace $w + V$. This is sufficient for the application of Theorem~\ref{thm:gs_stationary}, because all assumptions
there are stated in terms of $\Jcal \circ \InexactArgmin$,
which is invariant under the specific choice.

\begin{lemma}\label{lem:exact_min_continuous}
    Let $\Jcal$ be block-separable nonsmooth,
    and let the decomposition be induced by the block-structure,
    i.e., \eqref{eq:decomposition_by_blocks}.
    Then the exact minimization operator $\InexactArgmin = \ExactArgmin$ satisfies the assumptions
    of Theorem~\ref{thm:gs_stationary} with $\ContinuousArgmin = \ExactArgmin$.  In particular,
    the minimal energy $\Jcal \circ \ExactArgmin$ is continuous
    on $\op{dom} \Jcal$.
\end{lemma}
\begin{proof}
    The operator $\Mcal = \tilde{\Mcal} = \ExactArgmin$ is stable and monotone by construction.
    To show continuity of $\Jcal \circ \ExactArgmin$, note that
    the block-separability of the functional
    \begin{align*}
        \Jcal(v) = \Jcal_0(v) + \sum_{k=1}^M \varphi_k (R_k v)
    \end{align*}
    implies that its domain takes the form
    \begin{align*}
        \op{dom} \Jcal = \op{dom} \varphi_k \times \Bigl(\prod_{i\neq k} \op{dom} \varphi_i\Bigr).
    \end{align*}
    The assertion then follows from Corollary~\ref{cor:product_domain_continuity}.
\end{proof}

If the subspace decomposition is not induced
by the block-separable structure of $\Jcal$, then the minimization \eqref{eq:exact_minimization} is
in general not continuous.
We refer to \cite{GraeserSander2014} for an example.
The following result shows that continuity still holds if the domain is polyhedral.
It is a generalization of a result
shown in \cite{GraeserSander2014} for the case of a convex~$\Jcal$.

\begin{lemma}
    \label{lem:polyhedron_exact_minimization}
    Let $\op{dom} \Jcal$ be a convex polyhedron.
    Then $\InexactArgmin = \ContinuousArgmin = \ExactArgmin$ satisfies the assumptions
    of Theorem~\ref{thm:gs_stationary}.
\end{lemma}
\begin{proof}
    Monotonicity and stability are again given by construction,
    while continuity for polyhedral domains is shown in
    Corollary~\ref{cor:polyhedral_domain_continuity}.
\end{proof}

We can also show that inexact versions of $\ExactArgmin$
satisfy these assumptions if they guarantee sufficient descent.

\begin{lemma}
    \label{lem:polyhedron_inexact_minimization}
    Let the functional $\Jcal$ and the subspace decomposition
    satisfy either the assumptions of
    Lemma~\ref{lem:exact_min_continuous}
    or of Lemma~\ref{lem:polyhedron_exact_minimization}.
    Assume that the operator
    \begin{align*}
        \InexactArgmin:\op{dom}\Jcal \to \op{dom} \Jcal, \qquad
        \InexactArgmin - \Id : \op{dom}\Jcal \to V
    \end{align*}
    satisfies the sufficient descent condition
    \begin{align*}
        \Jcal(w)-\Jcal(\InexactArgmin(w)) \geq \varepsilon \bigl[\Jcal(w)-\Jcal(\ExactArgmin(w)) \bigr]
    \end{align*}
    for a fixed $\varepsilon > 0$.
    Then $\InexactArgmin$ satisfies the assumptions of Theorem~\ref{thm:gs_stationary}.
\end{lemma}
\begin{proof}
    By continuity of $\Jcal$ on $\op{dom} \Jcal$, for any $w\in \op{dom} \Jcal$ there exists
    a $\tilde{w} \in \op{dom} \Jcal \cap w + V$ with
    \begin{align*}
        \Jcal(\tilde{w}) = (1-\varepsilon) \Jcal(w) + \varepsilon \Jcal(\ExactArgmin(w)).
    \end{align*}
    Now we set $\ContinuousArgmin(w) = \tilde{w}$. Then we have
    \begin{align*}
        \Jcal(\InexactArgmin(w)) \leq (1-\varepsilon)\Jcal(w) + \varepsilon \Jcal(\ExactArgmin(w)) = \Jcal(\ContinuousArgmin(w)) \leq \Jcal(w),
    \end{align*}
    which is the required monotonicity.
    Continuity of $\Jcal$ and of $\Jcal \circ \ExactArgmin$ (Lemma~\ref{lem:exact_min_continuous})
    imply continuity of $\Jcal \circ \ContinuousArgmin$.
    Furthermore $\Jcal(\ContinuousArgmin(w)) = \Jcal(w)$ implies $\Jcal(\ExactArgmin(w)) = \Jcal(w)$
    which shows stability.
\end{proof}

\subsection{Polyhedral Gau\ss--Seidel}
We now consider the case of a block-separable nonsmooth functional
\begin{align*}
    \Jcal(v) = \Jcal_0(v) + \sum_{k=1}^M \varphi_k (R_k v)
\end{align*}
where, additionally, each $\varphi_k$ is piecewise smooth on a partition
of its domain into a finite set a convex polyhedra.
Furthermore we assume the induced subspace decomposition
\eqref{eq:decomposition_by_blocks}.
Here, solving the piecewise smooth subproblems in
each subspace $V_k$ can be nontrivial problem.

For the case that $\op{dom} \varphi_k$ is a simplex, \cite{KornhuberKrause2006}
proposed to further split $V_k$ into the one-dimensional spaces spanned by the simplex edges
and to successively minimize in each of those spaces.
This was generalized in \cite{GraeserSander2014} to general polyhedral partitions,
by using spaces
\begin{align*}
    V_k = \sum_{i=1}^{m_k} V_{k,i},
\end{align*}
where the one-dimensional subspaces $V_{k,i}$
are aligned with the edges of the polyhedral
partition of $\op{dom} \varphi_k$.
For the precise definition of these subspaces and the
treatment of the case that the edges do not span the
whole space we refer to \cite{GraeserSander2014}.

Instead of showing for each $V_k$ that this procedure takes
the form of an inexact minimization operator $\InexactArgmin_k$
in $V_k$ in the sense of Theorem~\ref{thm:gs_stationary}, we can
simply apply this theorem to the decomposition
\begin{align}
\label{eq:pgs_decomposition}
    \R^n = \sum_{k=1}^m \sum_{i=1}^{m_k} V_{k,i}
\end{align}
and use an exact or inexact minimization step in each one-dimensional
subspace $V_{k,i}$.
Such inexact minimization operators
$\InexactArgmin_{k,i}$ for $V_{k,i}$ can be constructed
by straightforward application of bisection with
fully practical termination criteria, see, e.g., \cite{GraeserSander2014}.
The resulting relaxation method is called the
\emph{Polyhedral Gau\ss--Seidel} method.  A variant of the following theorem
has been the main result of~\cite{GraeserSander2014}.

\begin{theorem}
 Let $\Jcal : \R^n \to \R \cup \{ \infty \}$ be coercive, proper, lower-semicontinuous, and
 continuous on its domain.  Suppose that there exists a finite partition of $\op{dom} \varphi$
 into nondegenerate, convex, closed polyhedra such that $\varphi$ is piecewise smooth
 with respect to this partition.
 Consider the decomposition of $\R^n$ into one-dimensional subspaces $V_{k,i}$ such that
 any tangent cone of the partition is generated by vectors from the one-dimensional spaces $V_{k,i}$.

 On each such subspace, let $\Mcal_{k,i}$ be the exact minimization
 operator~\eqref{eq:exact_minimization} or its inexact cousin in the sense of
 Lemma~\ref{lem:polyhedron_inexact_minimization}.
 Then all accumulation points of the TNNMG method are stationary.  Additionally,
 if $\Jcal$ is convex and has a unique minimizer, then the method will converge to that minimizer.
\end{theorem}

\begin{proof}
The result follows directly from Theorem~\ref{thm:gs_stationary} and Corollary~\ref{cor:gs_convergence},
if we can show continuity of $\Jcal \circ \Mcal_{k,i}$ and compatibility of the
decomposition~\eqref{eq:pgs_decomposition}.
The continuity of $\Jcal \circ \InexactArgmin_{k,i}$ was shown in
Lemmas~\ref{lem:polyhedron_exact_minimization} and~\ref{lem:polyhedron_inexact_minimization}
for the exact and the inexact case, respectively.  The compatibility of the splitting
with $\Jcal$ was shown in Lemma~5.1 of~\cite{GraeserSander2014}.
\end{proof}

\subsection{First-order models}

In this section we assume that $\Jcal$ is block-separable nonsmooth
and that the decomposition is induced by the block-structure,
i.e., \eqref{eq:decomposition_by_blocks} holds.
Then the local subproblems~\eqref{eq:local_subproblem}
take the form
\begin{align*}
    v^* \in \R^N: \qquad f(v^*) \leq f(v) \qquad \forall v \in \R^N
\end{align*}
where
\begin{align}
    \label{eq:model_restricted_functions}
    f(v) = f_w(v) = \underbrace{\Jcal_0(w+P v)}_{\equalscolon f_0(v)}
                  + \underbrace{\varphi_k(R_kw + v)}_{\equalscolon \psi(v)},
\end{align}
and $P = P_k$ is the prolongation operator defined in~\eqref{eq:prolongation_operators}.
Under the assumption that $f$ is convex
we will now construct inexact solvers for these problems
by solving approximate problems exactly.
To this end we introduce the notion of first-order dominating models.
These models differ from standard first-order models by
only approximating the smooth part $f_0$ of the functional $f$,
while the nonsmooth part $\psi$ is treated exactly.

Since we only need these models for incremental
problems we can simplify the situation without loss of generality
by only considering models fitted to $f$ in the origin.

\begin{definition}
    A functional $M:\R^N \to \Rinfty$ is called a
    \emph{first-order model}
    for $f$ at $0 \in \op{dom} \psi$ if there is a continuously differentiable $M_0:\R^N \to \R$ with
    \begin{align*}
        M(v) = f_0(0) + \inner{f_0'(0)}{v} + M_0(v) + \psi(v)
    \end{align*}
    such that $M_0(0) = 0$ and $M'_0(0) = 0$.
\end{definition}

First-order models allow to detect minimizers, i.e., if the original
functional is not minimal in $0$, then neither is the model.

\begin{lemma}
    \label{lem:model_decrease}
    Let $M$ be a first-order model of $f$ at $0 \in \R^N$,
    and $v \in \R^N$ with $f(v)<f(0)$. Then there is
    a $t>0$ with $M(tv) < M(0) = f(0)$.
\end{lemma}
\begin{proof}
    From convexity of $f$ we get for any $t \in (0,1)$ that
    \begin{align*}
        0 > f(v)-f(0)
        &\geq \frac{1}{t}(f(tv)-f(0)) \\
        &\geq \frac{1}{t}\bigl[ \inner{f_0'(0)}{tv} +\psi(tv) - \psi(0) \bigr].
    \end{align*}
    As $M_0$ is differentiable with $M_0'(0)=0$ we get for sufficiently small $t>0$ that
    \begin{align*}
        \frac{1}{2}\underbrace{(f(0)-f(v))}_{>0}
        \geq
            \frac{1}{t}\bigl[M_0(tv) + M_0(0)\bigr].
    \end{align*}
    Adding both inequalities and multiplying by $t>0$ yields
    \begin{equation*}
        0> \frac{t}{2} \bigl( f(v)-f(0) \bigr)
        \geq  M(tv) - M(0).
        \qedhere
    \end{equation*}
\end{proof}

Although a first-order model can only be minimal if the
exact functional is minimal, minimizing the model does not necessarily decrease the
value of the original functional. In order to obtain this
property the model must also be \emph{dominating}.

\begin{definition}
    A first-order model $M:\R^N \to \Rinfty$ for $f$ at
    $0 \in \op{dom} \psi$
    is called \emph{dominating} if it satisfies
    $M(v) \geq f(v)$ for all $v\in \R^N$.
\end{definition}

We will now show that exact minimization of a dominating first-order
model leads to an inexact minimization operator $\InexactArgmin$ in the sense
of Theorem~\ref{thm:gs_stationary}. Here, the exact minimization of the model
will play the role of the continuous minimization operator $\ContinuousArgmin$.
We define
\begin{align}\label{eq:inexact_min_by_model}
    \InexactArgmin^\text{mod}(w) \colonequals w + P \argmin_{v \in \R^N} M(v).
\end{align}
We will only consider the case that $M_0$ does not depend on $w$.
The following results can be generalized to $w$-dependent $M_0$ under strong technical
continuity assumptions for $M_0'$ with respect to $w$.

\begin{theorem}
    For any $w \in \op{dom} \Jcal$ let $f_w$ be as defined
    in~\eqref{eq:model_restricted_functions}.
    Set
    \begin{align*}
        M_w(v) \colonequals f_{w,0}(0) + \inner{f_{w,0}'(0)}{v} + M_0(v) + \psi(v)
    \end{align*}
    for some fixed function $M_0$, and assume that $M_w$ is a first-order dominating
    model for $f_w$ in each $w \in \op{dom} \Jcal$.
    Furthermore assume that $M_{0}'$
    is uniformly monotone, i.e.,
    there are constants
    $\alpha>0$ and $p>1$ such that
    \begin{align*}
        \alpha \|u-v\|^{p} \leq \inner{M_{0}'(u) - M_{0}'(v)}{u-v} \qquad \forall u,v \in \R^N.
    \end{align*}
    Then $\InexactArgmin^\text{mod}$ as defined in \eqref{eq:inexact_min_by_model}
    is an inexact minimization operator in the sense of
    Theorem~\ref{thm:gs_stationary}.
\end{theorem}
\begin{proof}
    By construction any dominating model leads to a monotone $\InexactArgmin^\text{mod}$,
    and Lemma~\ref{lem:model_decrease} shows that first-order models
    lead to stable $\InexactArgmin^\text{mod}$.
    Now we show continuity of $\InexactArgmin^\text{mod}$ and hence of $\Jcal \circ \InexactArgmin^\text{mod}$.

    For $i=1,2$ let $w_i \in \op{dom} \Jcal$ and $v_i = \argmin_{v \in \R^N} M_{w_i}(v)$.
    Then we have
    \begin{align*}
        \big\langle{M_{0}'(v_i) + f_{w_i,0}'(0)},{v - v_i}\big\rangle + \psi(v) - \psi(v_i) \geq 0.
    \end{align*}
    Testing this variational inequality for $i=1,2$ with $v_j$, $j\neq i$,
    adding the results, and using monotonicity gives
    \begin{align*}
        \alpha \|v_1-v_2\|^p \leq
        \big\langle{M_{0}'(v_1)-M_{0}'(v_2)},{v_1-v_2}\big\rangle
            &\leq \|f_{w_1,0}'(0) - f_{w_2,0}'(0)\| \|v_1-v_2\| \\
            &\leq \|P\| \|\Jcal_0'(w_1) - \Jcal_0'(w_2)\| \|v_1-v_2\|.
    \end{align*}
    Dividing by $\|v_1-v_2\|$ and exploiting
    continuity of $\Jcal_0'$ we find that
    $\argmin_{v \in \R^N} M_{w}(v)$
    and thus $\InexactArgmin$ depends continuously on $w$.
\end{proof}

\bigskip

The theory of first-order dominating models is useful because it is frequently easy to construct
models that are much easier to minimize than the actual functional, while at the same time
producing sufficient descent to act as a nonlinear smoother for a fast multigrid method.
In the following we briefly discuss how such models can be constructed.

The proof of Lemma~\ref{lemma:quadratic_dominating} below requires the following sufficient
condition, which is of independent interest.
\begin{lemma}
    \label{lemma:dominating}
    Let $M$ be a first-order model of $f$ at $0\in \R^N$, and
    \begin{align*}
        \inner{f_0'(v)-f_0'(0)}{v} \leq \inner{M_0'(v)-M_0'(0)}{v} \qquad \forall v \in \R^N.
    \end{align*}
    Then $M$ is also a dominating model.
\end{lemma}
\begin{proof}
    For $v \in \R^N$ we have
    \begin{align*}
            f(v)
                &= \psi(v) + f_0(0) + \inner{f_0'(0)}{v} + \int_0^1 \frac{1}{t} \inner{f_0'(tv)-f_0'(0)}{tv} \, dt \\
                &\leq \psi(v) + f_0(0) + \inner{f_0'(0)}{v} + \int_0^1 \inner{M_0'(tv)}{v} -\inner{M_0'(0)}{v} \, dt
                = M(v),
    \end{align*}
    which is the assertion.
\end{proof}

With this result we can construct models with a quadratic smooth part.

\begin{lemma}
    \label{lemma:quadratic_dominating}
    Let $B \in \R^{N \times N}$ be a symmetric positive semi-definite matrix.
    Then
    \begin{align*}
        M(v) = f_0(0) +  \inner{f_0'(0)}{v} + \frac{1}{2}\inner{B v}{v} + \psi(v)
    \end{align*}
    is a first-order model for $f$ at $0$.
    If $B$ satisfies
    \begin{align}
    \label{eq:condition_quadratic_model}
        \inner{f_0'(v)-f_0'(0)}{v} \leq \inner{Bv}{v} \qquad \forall v \in \R^N,
    \end{align}
    then $M$ is dominating.
\end{lemma}
\begin{proof}
    The fact that $M$ is a first-order model is trivial.
    If $B$ satisfies~\eqref{eq:condition_quadratic_model}, then
    $M_0'(v)-M_0'(0) = Bv$, and Lemma~\ref{lemma:dominating} implies that $M$ is dominating.
\end{proof}

\begin{example}\label{example:quadratic_smooth}
    Let $f(v)=\frac{1}{2}\inner{Av}{v} - \inner{b}{v} + \psi(v)$
    and $\inner{Av}{v} \leq \inner{Bv}{v}$ for all $v \in \R^N$
    for symmetric positive definite matrices $A,B$.
    Then
    \begin{align*}
        M(v) = \frac{1}{2}\inner{Bv}{v} - \inner{b}{v} + \psi(v)
    \end{align*}
    is a dominating  first-order model for $f$ at $0$.
\end{example}

Models of this type are interesting, because approximating matrices $B$ can be easy
to construct as long as $A$ is small. The following radical choice can be applied
to plasticity and friction problems, which are represented by Example~\ref{ex:norm_nondifferentiability}.

\begin{example}
    \label{example:rotationally_symmetric_model}
    Let $f$ and $M$ be as in Example~\ref{example:quadratic_smooth}
    with $\psi(v) = h(\|v_0+v\|)$, and $B=\alpha I$ with $\alpha \geq \lambda_{\max}(A)$.
    Then $M$ can be written as
    \begin{align*}
            M(v) &= \frac{\alpha}{2}\|v\|^2 - \inner{b}{v} + \psi(v) \\
                 &= \frac{\alpha}{2}\|v_0+v\|^2 - \inner{b + \alpha v_0}{v_0+v}  + \frac{\alpha}{2}\|v_0\|^2 + \inner{b}{v_0} + \psi(v) \\
                 &= \tilde{h}(\|v_0+v\|) - \inner{\tilde{r}}{v_0+v} + \op{const},
    \end{align*}
    with $\tilde{h}(t) = \frac{\alpha}{2}t^2+h(t)$ and
    $\tilde{r}=b+\alpha v_0$.
\end{example}
%\todograeser{@OS: Bsp und Lemma verbinden?}

In this example $M$ is rotationally symmetric with respect to $v-v_0$.
Hence the minimizer $u$ of $M$ can be computed by
solving a scalar minimization problem
on the line $\{ t \tilde{r} - v_0 \st t \in \R \} \subset \R^N$.

\begin{lemma}
    Let $h:\R \to \Rinfty$, $r\in \R^N \setminus \{0\}$
    and $v_0 \in \R^N$.
    Then any minimizer $u$ of
    \begin{align*}
        g:\R^N \to \Rinfty, \qquad
        g(v)=h(\|v_0+v\|) - \inner{r}{v_0+v}
    \end{align*}
    satisfies $u = tr-v_0$ for $t>0$.
\end{lemma}
\begin{proof}
    Let $u$ be a minimizer and assume that $u$ is not of the form $u = tr - v_0$.
    Then $|\inner{r}{v_0+u}| < \|r\|\|v_0+u\|$.
    Now let $\tilde{u} = \frac{\|v_0+u\|}{\|r\|}r - v_0$. Then we have
    \begin{align*}
        g(\tilde{u})
            &= h(\|v_0+u\|) - \|r\|\|v_0+u\| \\
            &< h(\|v_0+u\|) - |\inner{r}{v_0+u}|
        \leq h(\|v_0+u\|) - \inner{r}{v_0+u} = g(u),
    \end{align*}
    which contradicts the assumption.
\end{proof}

Another example where the simplification proposed in Example~\ref{example:rotationally_symmetric_model}
can be helpful are problems with local simplex constraints
as introduced in Example~\ref{ex:vector_valued_phasefield}.

\begin{example}
    \label{example:simplex_constrained_model}
    Let $f$ and $M$ be as in Example~\ref{example:quadratic_smooth}
    with $B=\alpha I$, $\alpha \geq \lambda_{\max}(A)$,
    and $\psi(v) = \chi_G(v_0+v)$ where $\chi_G$ is the indicator
    function of the Gibbs simplex $G$ defined in \eqref{eq:gibbs_simplex}.
    Similar to Example~\ref{example:rotationally_symmetric_model} the
    dominating first-order model $M$ can now be written as
    \begin{align*}
        \begin{split}
            M(v)
%                &= \frac{\alpha}{2}\|v\|^2 - \inner{b}{v} + \psi(v) \\
%                 &= \frac{\alpha}{2}\|v_0+v\|^2 - \inner{b + \alpha v_0}{v}  - \frac{\alpha}{2} \|v_0\|^2 + \psi(v) \\
%                 &= \frac{\alpha}{2}\|v_0+v\|^2 - \inner{b + \alpha v_0}{v_0+v}  + \frac{\alpha}{2}\|v_0\|^2 - \inner{b}{v_0} + \psi(v) \\
%                 &= \frac{\alpha}{2}\|v_0+v\|^2
%                 - \inner{b + \alpha v_0}{v_0+v}  + \op{const} + \psi(v)\\
%                 &= \frac{\alpha}{2}\|v_0+v\|^2
%                 - \inner{\tilde{r}}{v_0+v} + \chi_G(v_0+v) + \op{const}
                 &= \frac{\alpha}{2}\|v_0+v - \tilde{r}\|^2 + \chi_G(v_0+v) + \op{const}
        \end{split}
    \end{align*}
    with $\tilde{r}=\tfrac{1}{\alpha}b+v_0$.
    Hence the minimizer $u$ of $M$ can be written as $u = u_0 - v_0$
    where $u_0 = \argmin_{z \in G} \|z-\tilde{r}\|^2$
    is the Euclidean projection of $\tilde{r}$ into $G$.
    This is an important simplification, because this projection
    can be computed exactly with a simple algorithm in $O(N \log(N))$ time
    \cite{graeser_sander:2014, wang_carreira:2013}.
\end{example}

The previous examples assumed that the smooth part
$\Jcal_0$ of the functional is quadratic, and replaced
it by a simpler quadratic functional.
For problems like Example~\ref{example:anisotropic_problem} where this is not the case,
using local first-order models can still be beneficial.
The construction of such models is based on the following direct consequence of the chain rule.
\begin{lemma}
    \label{lemma:quasilinear_model_bound}
    Let the smooth part $\Jcal_0$ of the global functional take the form
    \begin{align*}
        \Jcal_0(v) = \sum_{i=1}^{l} \gamma_i (D_iv)
    \end{align*}
    with $D_i \in \R^{d \times n}$
    and $\gamma_i: \R^d \to \R$ convex and continuously
    differentiable for $i=1,\dots,l$.
    Assume that each $\gamma_i'$ is Lipschitz continuous
    with a Lipschitz constant $L_i$. Then
    \begin{align*}
        \inner{\Jcal_0'(x) - \Jcal_0'(y)}{x-y}
        \leq \inner{\Bcal (x-y)}{x-y}
        \qquad \forall x,y \in \R^n
    \end{align*}
    for the symmetric positive semi-definite matrix
    $\Bcal = \sum_{i=1}^l L_i D_i^T D_i \in \R^{n \times n}$.
\end{lemma}

\begin{example}
    \label{example:anisotropic_model}
    Let the global problem take the form given in Example~\ref{example:anisotropic_problem}
    with $\Jcal_0$ as in Lemma~\ref{lemma:quasilinear_model_bound}.
    For the local problem consider the first-order
    model given in Lemma~\ref{lemma:quadratic_dominating}
    with $B = P^T \Bcal P$.
    Then Lemma~\ref{lemma:quasilinear_model_bound}
    together with the chain rule shows that
    \begin{align*}
        \inner{f_0'(x) - f_0'(y)}{x-y} \leq \inner{B(x-y)}{x-y}
        \qquad \forall x,y \in \R^N.
    \end{align*}
    Hence, by Lemma~\ref{lemma:quadratic_dominating} the model
    is also dominating.
    In contrast to $f_0$ the functional $M_0$ is quadratic
    such that the minimizer of $M$ can be directly computed.
    For problems with more general nonsmooth terms $\varphi_k$
    where this is not the case, using an iterative method
    to minimize $M$ will often be much faster
    because $M_0$ is cheaper to evaluate.
\end{example}

When using the first-order models presented here as the basis
for an iterative method to actually solve the local minimization
problem for $f$, the resulting
algorithms are variants of so-called proximal methods.
In case of the Examples~\ref{example:rotationally_symmetric_model}
and~\ref{example:simplex_constrained_model} these are proximal
gradient type methods, while the model in Example~\ref{example:anisotropic_model}
would lead to a proximal Newton-type method.
In the case that the nonsmooth part $\psi$ of $f$ is
the indicator function of a convex set while the model
$M_0$ for the smooth part is quadratic with a matrix $B=\alpha I$,
such methods can also be viewed as projected
gradient-type methods.

We emphasize that we do not propose and analyze the resulting
proximal methods here, but only discuss their use as inexact
local solvers inside of a nonlinear block Gau\ss--Seidel smoother.
For a discussion of proximal gradient and proximal Newton-type iterations
we refer to~\cite{LeeSunSaunders2014} and the references therein.

\section{Truncated Newton corrections}
\label{sec:coarse_grid_corrections}

\subsection{Constructing suitable correction spaces}

The second step of the TNNMG iteration consists of an inexact Newton step in an iteration-dependent
subspace $W_\nu$ of $\R^n$ that is constructed such that all necessary derivatives of $\Jcal$
are well-defined. In principle $W_\nu$ can be defined as the largest subspace of $\R^n$ such that there is an
open ball $B_\varepsilon(u^{\nu+\frac12})$ of radius $\varepsilon>0$ around $u^{\nu+\frac12}$ such that
$\Jcal$ is twice continuously differentiable on
\begin{align}
\label{eq:def_of_w_nu}
    (u^{\nu+\frac12} + W_\nu) \cap B_\varepsilon(u^{\nu+\frac12}).
\end{align}
However, depending on the problem, the practical construction of this subspace can be technical.

In many cases, the implementation can be simplified by using a correction space $W_\nu$ that is smaller
than possible.
The practical behavior of the TNNMG method may be unaffected when $W_\nu$ is replaced by a slightly
smaller space, and the convergence results hold for all spaces $W_\nu$ that allow for a
well-defined Newton problem.

Let the energy functional $\Jcal$ be block-separable nonsmooth i.e.,
\begin{equation*}
 \Jcal(v)
 =
 \Jcal_0(v) + \sum_{k=1}^M \varphi_k(R_k v),
 \qquad
 v \in \R^n,
\end{equation*}
with a $C^2$ functional $\Jcal_0$.
Then it is clear, that the space $W_\nu$ can be defined as a product space
\begin{equation}
 \label{eq:trunated_space_block_structure}
 W_\nu
 \colonequals R^{-1} \prod_{k=1}^M
    W_{\nu,k}
\end{equation}
where $W_{\nu,k}\subset \R^{N_k}$ is a local subspace such that $\varphi_k$ is locally smooth
near $R_ku^{\nu+\frac12}$ in the above given sense.
Here we make use of the isomorphism $R:\R^n \to \prod_{k=1}^M \R^{N_k}$
translating between $\R^n$ and the block representation according
to Definition~\ref{def:block_separable} such that the $k$-th block of $v \in \R^n$
is given by $R_k v \in \R^{N_k}$.

The easiest way to construct a suitable correction space $W_\nu$
is by disabling entire blocks.  For a given pre-smoothed iterate $u^{\nu + \frac{1}{2}}$ define
the set of ``inactive'' blocks
\begin{multline}
    \label{eq:inactive_blocks}
 \mathcal{N}_\nu^\circ
 \colonequals
 \Big\{ k =1,\dots,M
     \st
     \text{$\varphi''_k(R_k(\cdot))$ exists} \\
     \text{and is continuous on $B_\varepsilon(R_k u^{\nu+\frac12})$ for some $\epsilon>0$} \Big\}.
\end{multline}
Then frequently a reasonable correction space is
\begin{equation}
\label{eq:def_truncated_space}
 W_\nu
 \colonequals
 R^{-1}\prod_{k=1}^M
 \begin{cases}
  \R^{N_k}  & \text{if $k \in \mathcal{N}_\nu^\circ$}, \\
  \{ 0\}^{N_k} & \text{otherwise}.
 \end{cases}
\end{equation}
When all blocks are one-dimensional, this definition even yields the optimal space.
It also produces the optimal space for functionals as in
Example~\ref{ex:norm_nondifferentiability}, where the nonsmoothness is a block-wise norm function
\begin{equation*}
 \varphi_k(v) \colonequals \lVert v \rVert
 \qquad
 \forall v \in \R^{N_k}.
\end{equation*}

\begin{figure}
 \begin{center}
  \begin{overpic}[width=0.5\textwidth]{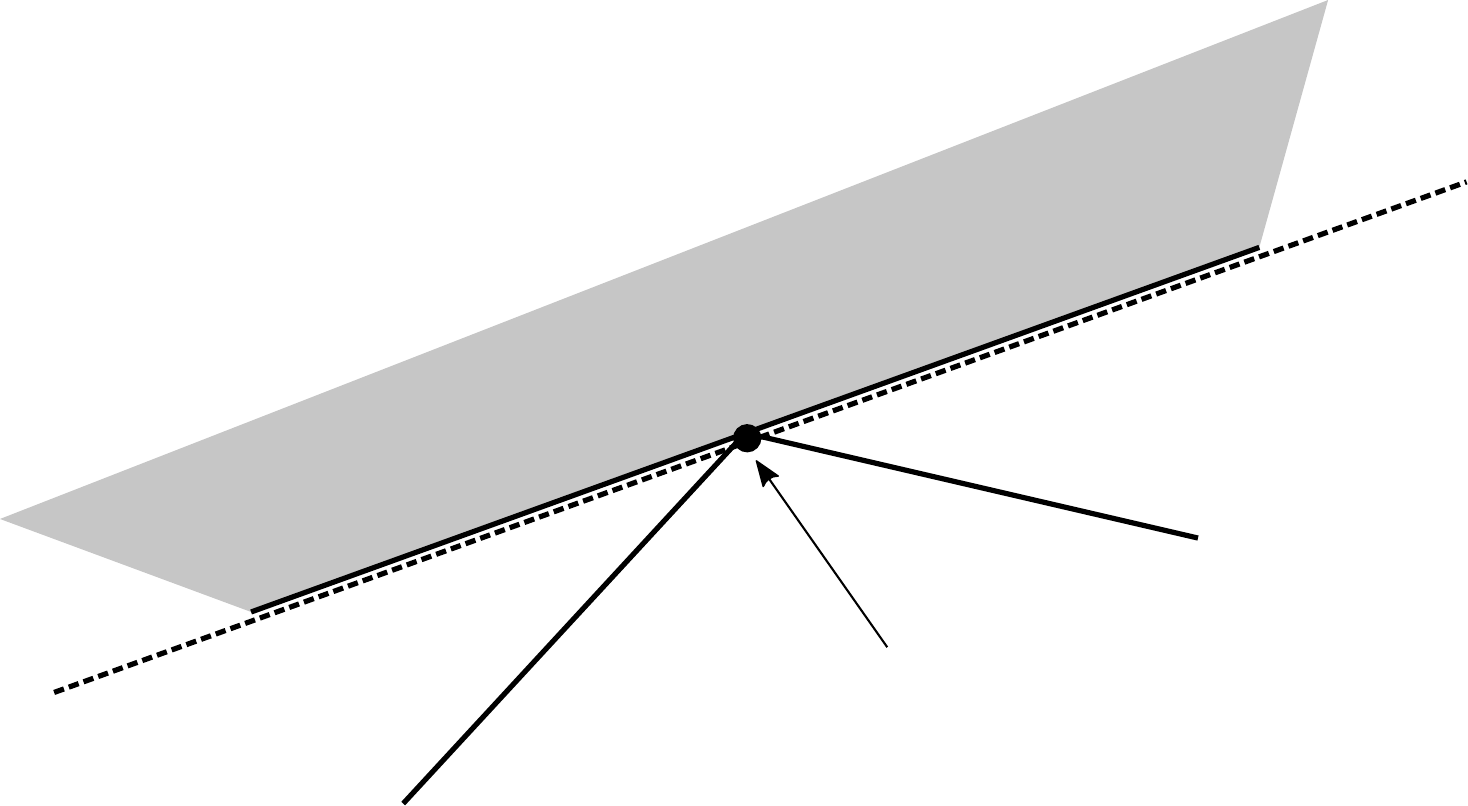}
   \put(60,5){vertex position $u^{\nu + \frac{1}{2}}$}
   \put(22,28){obstacle}
   \put(5,2){$W_\nu$}
   \put(93,33){$W_\nu$}
  \end{overpic}
  \hspace{0.08\textwidth}
  \begin{overpic}[width=0.4\textwidth]{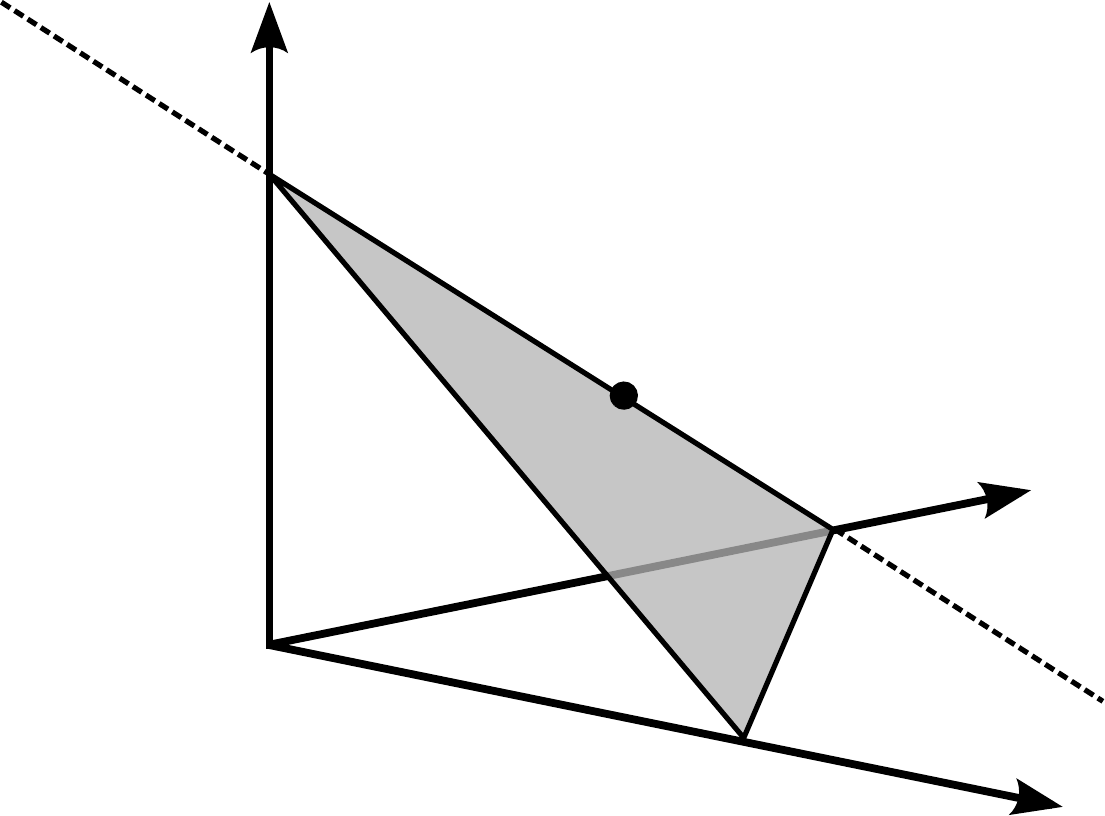}
   \put(60,40){$u^{\nu + \frac{1}{2}}$}
   \put(60,22){$G$}
   \put(5,58){$W_\nu$}
   \put(95,15){$W_\nu$}
  \end{overpic}
 \end{center}
 \caption{Constructing the correction spaces $W_\nu$.
   Left: Contact problem. At a grid vertex in contact, $W_\nu$ can contain the tangential direction.
   Right: Simplex-constrained problem. The correction spaces can contain the affine span
   of simplex faces.}
 \label{fig:nonsimple_correction_spaces}
\end{figure}
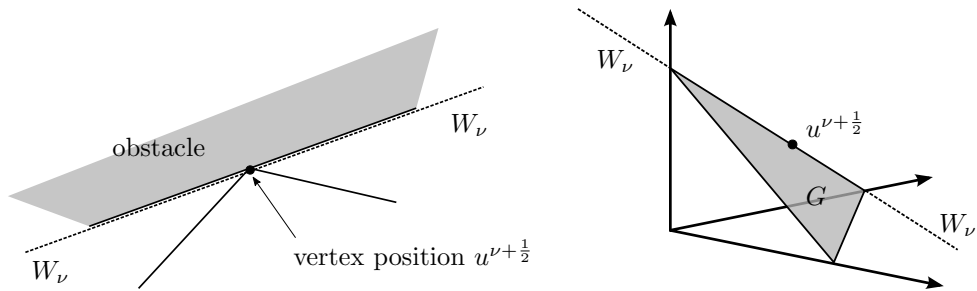

In other cases $\Jcal'$ and $\Jcal''$ may exist in larger spaces than the one
defined by~\eqref{eq:def_truncated_space}. As an example, consider a contact problem in
linear elastomechanics. There, it can be convenient to group the degrees of freedom in
$d$-dimensional blocks, corresponding to the Lagrange nodes of the finite element space.
For each $k=1,\dots,M$, the nonsmooth function $\varphi_k$ is then the indicator functional
of a half-space in $\R^d$, which models the restriction on the normal displacement of the $k$-th Lagrange node.
If the degrees of freedom in a block $k$ are in contact, the naive
approach~\eqref{eq:def_truncated_space} disables that block
completely. On the other hand, $\varphi_k$ is still $C^2$ in the tangential plane
of the half-space, and a bigger space $W_\nu$ could therefore be constructed by including these
tangent directions~\cite{GraeserSackSander2009} (Figure~\ref{fig:nonsimple_correction_spaces}, left).
Such a space then includes sliding along the obstacle in the linear correction step.

Vector-valued phase-field models as described in Example~\ref{ex:vector_valued_phasefield}
form a more extreme example.  In this example
the degrees of freedom are grouped in blocks of dimension $N^k=L$, and the nondifferentiability
$\varphi_k$ is the indicator functional $\chi_G$ of the Gibbs simplex $G$~\eqref{eq:gibbs_simplex}.
Here, block-wise truncation as in~\eqref{eq:def_truncated_space}
would truncate the whole space, since $\chi_G$ is nowhere differentiable.
Hence the linear correction would entirely be removed and the
method would degenerate to the pure nonlinear smoother.
Truncating the block only if the solution is on the relative
boundary of $G$ and keeping the tangent space to $G$
otherwise is still not a good choice, because solutions typically
live on the boundary of $G$ in almost every block $k=1,\dots,M$.
Again this would effectively result in an almost complete removal
of the linear correction.

Fortunately, for this example the largest possible correction space $W_\nu$
is much larger compared to the resulting space in both of these
simple constructions.
To really construct $W_\nu$ we consider
the active sets
\begin{equation*}
    \mathcal{N}_{\nu,k}^\bullet \colonequals \Big\{j = 1,\dots, L \st (R_k u^{\nu + \frac{1}{2}})_{j}=0\Big\}
\end{equation*}
for each block $R_k u^{\nu + \frac 12} \in G$.
Then $\Jcal$ is differentiable near $u^{\nu + \frac 12}$ in the space
\begin{align}
\label{eq:correction_space_for_phase_fields}
 W_\nu
 \colonequals
    R^{-1}\prod_{k=1}^M
    \op{span}\Bigl\{
        \eta^{ij} \in \Ecal \bigst
        i,j \notin \mathcal{N}_{\nu,k}^\bullet
    \Bigr\},
\end{align}
where
\begin{align*}
    \Ecal = \Big\{
        \eta^{ij}=e^i-e^j \in \R^L \bigst
        1\leq i < j \leq L \Big\}
\end{align*}
is the set of edges of $G$.
It is easy to see that this space is in fact the maximal subspace
of $\R^n$ where $\Jcal(u^{\nu + \frac 12} + \cdot)$ is differentiable near the origin:
On the one hand restriction of the nonsmooth term $\varphi(u^{\nu + \frac 12} + \cdot)$
to $W_\nu$ is constant in a ball around the origin such that the functional is
locally $C^2$.
On the other hand $\varphi(u^{\nu + \frac 12} + \cdot)$
has a jump from $0$ to $\infty$ at the origin along any other direction.

\begin{remark}
 Unfortunately, in an implementation, the decision of whether a degree of freedom is active or not
 can only be taken in an approximate sense, by replacing a condition like $R_k u = 0$ by
 $\lvert R_k u\rvert \le \varepsilon$ due to problems caused by finite-precision arithmetic.
 This effectively replaces the arbitrarily small positive $\varepsilon$
 that appears in~\eqref{eq:def_of_w_nu} and~\eqref{eq:inactive_blocks}
 by a fixed one.
 The choice of $\varepsilon$ can be important: if $\varepsilon$ is too small, then the solver
 may become unstable.  If it is too large, the convergence rates deteriorate. This effectively
 introduces a parameter into the algorithm; however, in a finite element context, while this
 parameter depends on the boundary value problem that is being solved, it does not depend
 on the discretization.
\end{remark}

For some applications
continuous differentiability is not the only criterion that guides
the construction of the correction spaces $W_\nu$.
For example, more realistic examples of phase-field models include
a nonlinear energy whose derivatives become singular when approaching the boundary of
the admissible set~\cite{GraeserKornhuberSack2014b}.
Using a linearization of $\Jcal$ in a space constructed
as discussed above may lead to unbounded derivatives as some $u_k$ approach the singularity.
As a consequence the condition number of the linear system may become
arbitrarily large, independently of the spatial discretization parameter
leading to numerical problems and slow convergence of linear solvers.

To avoid this the general truncation strategy can be extended
by---additionally to twice continuous differentiability---%
requiring that the second derivatives do not become ``too large''
in the truncated ball~\eqref{eq:def_of_w_nu}.
In practice this can be achieved by also checking
the values of second-order derivatives or the distance to the
singularity in the construction of the space.
For example, in a multi-component phase-field model where
$\varphi_k$ takes the form
\begin{align*}
    \varphi_k(z) = \chi_G(z) + \sum_{j=1}^L \varphi_{k,j}(z_j)
\end{align*}
and each $\varphi_{k,j}'$ and $\varphi_{k,j}''$ is singular in zero,
this leads to active sets defined by
\begin{equation*}
 \mathcal{N}^\bullet_{\nu,k}
 \colonequals
    \Big\{j = 1,\dots,L \bigst (R_k u^{\nu + \frac 12})_{j}=0 \text{ or } \varphi_{k,j}''((R_ku^{\nu + \frac 12})_{j})>C\Big\}
\end{equation*}
for a large constant $C$.
These active sets are then used in the definition~\eqref{eq:correction_space_for_phase_fields}
of the correction space $W_\nu$. This modification can lead to a considerable
increase in efficiency.

Certain energies, such as the anisotropic fracture energy in~\cite{miehe_welschinger_hofacker_2010},
have a smooth energy contribution $\Jcal_0$ that is is not $C^2$ but only differentiable with a
locally Lipschitz-continuous derivative ($LC^1$). For such cases, the definition of the space
$W_\nu$ can be relaxed, by replacing second derivatives in the definition of $W_\nu$ by generalized
derivatives of $\Jcal'$. In this case the correction step \eqref{eq:inexact_newton_problem}
takes the form of a generalized Newton-step as proposed, e.g., in
\cite{penot:2012}.

\subsection{Algebraic representations}

In practical implementations of the TNNMG method, matrix and vector representations of
$\Jcal''(u^{\nu + \frac{1}{2}})|_{W_\nu \times W_\nu}$ and $\Jcal'(u^{\nu + \frac{1}{2}})|_{W_\nu}$
are needed, respectively. For simplicity, implementations are suggested to always work with matrices
and vectors in $\R^n$, and extend $\Jcal''(u^{\nu + \frac{1}{2}})|_{W_\nu \times W_\nu}$
and $\Jcal'(u^{\nu + \frac{1}{2}})|_{W_\nu}$ by zero to the orthogonal complement
of $W_\nu$ in $\R^n$.

This extension is simple if $W_\nu$ is constructed by the truncation of active
blocks~\eqref{eq:def_truncated_space}.  In this case, the vector representation of $\Jcal'(u^{\nu + \frac{1}{2}})|_{W_\nu}$
is simply the vector in $\R^n$ that contains the partial derivatives of $\Jcal$ for all
$k \in \mathcal{N}^\circ_\nu$ (which exist by construction of $\mathcal{N}_\nu^\circ$), and zero elsewhere.
Similarly, the matrix representation of $\Jcal''(u^{\nu + \frac{1}{2}})|_{W_\nu \times W_\nu}$ contains the second partial derivative in all
blocks where neither the column nor the row index is active, and zero otherwise.
This makes the matrix semi-definite.

\begin{remark}
 It would be easy to make the matrix positive definite by replacing the zeros on the diagonal by
 a fixed positive constant $\alpha$. However, this constant, if chosen badly, may influence the
 condition number of the linear correction system. If possible we prefer to modify the linear
 solver to handle the semi-definite problem directly (see Chapter~\ref{sec:linear_mg}).
\end{remark}

For the general case, denote by $\Pi_S$ the orthogonal projection onto a closed subspace $S$,
and let $\Jcal'(u^{\nu+\frac12})$ and $\Jcal''(u^{\nu+\frac12})$ be a generalized gradient and Hessian
of $\Jcal$ at $u^{\nu+\frac12} \in \R^n$, respectively.
Here we understand ``generalized'' in the sense that the entries are the
partial derivatives of $\Jcal$ wherever they exist, and arbitrary elsewhere.
If $\Jcal$ is block-separable nonsmooth, then,
by the block structure~\eqref{eq:trunated_space_block_structure}
of $W_\nu$, the global projection
$\Pi_{W_\nu} : \R^n \to W_\nu$ is
(up to the isomorphism $R$)
a block-diagonal matrix with diagonal blocks
\begin{align*}
    (R \Pi_{W_\nu}R^{-1})_{kk}
    = R_k \Pi_{W_\nu}R_k^{-1}
    = \Pi_{W_{\nu,k}}
    \in \R^{N_k \times N_k},
\end{align*}
where $\Pi_{W_{\nu,k}}$ is the projection onto $W_{\nu,k} = R_k W_\nu$.
For example, for simplex-constrained problems we get~\cite{graeser_sander:2014}
\begin{align*}
    \bigl(\Pi_{W_{\nu,k}}\bigr)_{ij} =
    \begin{cases}
        \delta_{ij}-\frac{1}{L-|\mathcal{N}_{\nu,k}^\circ(R_ku^{\nu+\frac12})|}
                 &\text{if $i,j \notin \mathcal{N}_{\nu,k}^\circ(w_k)$},\\
        0 &\text{otherwise}.
    \end{cases}
\end{align*}
Using $\Pi_{W_\nu}$ we can write the representations of the
truncated gradient and Hessian in terms of the coordinate system
in $\R^n$ as
\begin{align*}
    g_\nu &\colonequals \Jcal'(u^{\nu+\frac12})|_{W_\nu} = \Pi_{W_\nu}^T \Jcal'(u^{\nu+\frac12}), \\
    H_\nu &\colonequals \Jcal''(u^{\nu+\frac12})|_{W_\nu \times W_\nu} = \Pi_{W_\nu}^T \Jcal''(u^{\nu+\frac12}) \Pi_{W_\nu}.
\end{align*}
As $\Pi_{W_\nu}$ is block-diagonal $\Jcal''(u^{\nu+\frac12})|_{W_\nu \times W_\nu}$
inherits the sparsity pattern of $\Jcal''(u^{\nu+\frac12})$,
and all block-entries can be computed independently according to
\begin{align*}
    (R\Jcal''(u^{\nu+\frac12})|_{W_\nu \times W_\nu} R^{-1})_{ij}
    = \Pi_{W_{\nu,i}}^T \bigl(R\Jcal''(u^{\nu+\frac12}) R^{-1}\bigr)_{ij} \Pi_{W_{\nu,j}}.
\end{align*}
Using this matrix representation the solution of \eqref{eq:inexact_newton_problem}
is given by $v^\nu = -H_\nu^+ g_\nu$, where $(\cdot)^+$ is the Moore--Penrose pseudo-inverse,
and this is independent of the entries selected for non-existing partial derivatives.
See~\cite{graeser_sander:2014} for more details on TNNMG methods for vector-valued
phase-field problems.

\section{Inexact solution of linear subproblems}
\label{sec:linear_mg}

The final piece in the puzzle is how to produce suitable inexact solutions of the linear
correction systems~\eqref{eq:inexact_newton_problem}.
The convergence result in Theorem~\ref{thm:gs_stationary} only requires that the linear
correction shall not increase the energy, a feature that is always true due to the subsequent
line search step. However, for fast error reduction rates you need something that causes global information
exchange. Various alternatives exist; the best choice depends on the situation.

\subsection{Geometric multigrid iterations}

In many cases, the minimization problem for $\Jcal$ on $\R^n$ originates from the discretization of a partial
differential equation on a function space $\Scal_J$.
The default choice for an inexact solver for~\eqref{eq:inexact_newton_problem} is then a single $V$-cycle iteration of geometric
multigrid.  With this choice, the TNNMG algorithm can be interpreted as a multigrid method
with a nonlinear pre-smoother, a truncated linear coarse grid correction, and a particular line search.  One
expects multigrid convergence rates at least asymptotically, and these are indeed frequently
observed in practice~\cite{sander:2017,GraeserSackSander2009,GraeserKornhuber2009}.
Doing more than a single iteration, or a $W$-cycle iteration, may sometimes increase performance.

Implementations of geometric multigrid need minor modifications to cope with the restriction to the
truncated correction space $W_\nu$.
Suppose we have a hierarchy $\Scal_0 \subset \dots \subset \Scal_J$ of subspaces,
with the natural embedding of $\Scal_{k-1}$ into $\Scal_k$ given by the prolongation matrix $\Pcal_k \in \R^{\dim \Scal_k \times \dim \Scal_{k-1}}$.
We set $A_J \colonequals \Jcal''(u^{\nu + \frac 12})|_{W_\nu \times W_\nu}$ with the extension by zero to the orthogonal complement
of $W_\nu$ as described in the previous chapter.  From this we
construct a hierarchy of stiffness matrices by setting $A_{k-1} \colonequals \Pcal_k^T A_k \Pcal_k$.
We further assume the existence of linear smoothing operators $B_k : \R^{\dim \Scal_k} \to \R^{\dim \Scal_k}$.
A linear multigrid step for \eqref{eq:inexact_newton_problem} starting from a zero initial value
takes the form

\medskip

\begin{algorithm}[H]
\SetAlgoLined
 \KwIn{Given $u^{\nu+\frac{1}{2}}$}
 Set $r^J = - \Jcal'(u^{\nu + \frac{1}{2}})|_{W_\nu}$ \\
 \For{$k=J,\dots,1$} {
   Compute $\hat{v}^k = B_k r^k$ \\
   Set $r^{k-1} = \Pcal_k^T(r^k - A_k \hat{v}^k)$
  }
  Compute $\hat{v}^0 = B_0 r^0$ \\
 \KwOut{$\displaystyle v^\nu = \Pi_{W_\nu}\sum_{k=0}^J \Bigl(\prod_{l=k+1}^J \Pcal_l\Bigr) \hat{v}^k$}
\end{algorithm}

\medskip
In this simple form there is a single pre-smoothing step and no post-smoothing.
The extension
to a $V$-cycle with multiple pre- and post-smoothing steps is straightforward.

The first difference to a standard linear multigrid step
is the projection $\Pi_{W_\nu}$ in the last line.
It is needed since, while the method does in principle
act in the quotient space $\R^n/\op{ker}A_J=W_\nu$, it represents
all corrections in the larger space $\R^n$.
Since the prolongation operators are not aware of the kernel
they may create spurious contributions in $\op{ker}A_J$.
However the corrections in the quotient space $W_\nu$ are
invariant under those contributions and the latter can easily be removed
by the additional projection onto $W_\nu$.

The second difference is the choice of smoothers $B_k$.
Since the matrices $A_k$ are in general symmetric but positive semi-definite
only, standard choices may not work as expected.
Instead of approximating the inverse $A_k^{-1}$ the smoothers $B_k$ are
intended to approximate the pseudo-inverse matrices $A_k^+$ such that they
are invariant under any residual contribution in the kernel of the respective $A_k$.
In the following, we drop the level index $k$
for simpler notation and assume nested indices representing the block-structure of the matrices.
A common choice is a block Gau\ss--Seidel method based on a block-triangular
decomposition $A = D +L +L^T$, where $D_{ij} = \delta_{ij} A_{ij}$.
However, $D + L$ may be singular because the diagonal
entries $D_{ii}$ are in general only positive semi-definite.
Solving local systems involving the $D_{ii}$ therefore requires
a kernel-invariant local solver like, e.g., the CG method,
or subspace correction methods as in~\cite{LeeWuXuZikatanov2008}.

\begin{remark}
In certain situations, the following cheap modification can help to avoid numerical difficulties
due to the infinite condition number of $D_{ii}$.
For a small parameter $0<\alpha \ll 1$ set
\begin{align*}
    \bigl(\tilde{D}_{ii}\bigr)_{lm} \colonequals
    \begin{cases}
        (D_{ii})_{lm} & \text{if $l \neq m$}, \\
        \alpha & \text{if $l=m$ and $(D_{ii})_{lm}=0$}, \\
        \bigl((1+\alpha)D_{ii}\bigr)_{lm} & \text{if $l=m$ and $(D_{ii})_{lm}\neq 0$}.
    \end{cases}
\end{align*}
Then $\tilde{D}_{ii}$ is symmetric positive definite, and the resulting correction
obtained by a CG method for $\tilde{D}_{ii}$
can be viewed as a damped version of the one obtained for $D_{ii} + \Pi_{\op{ker} D_{ii}}$.
Numerical experiments in~\cite{graeser_sander:2014} showed that
the CG method was robust for $\alpha \geq 10^{-14}$,
and the overall convergence rates were hardly influenced as long as $\alpha \leq 10^{-4}$.
\end{remark}

The modified linear multigrid algorithm becomes particularly simple if the
restriction to the space $W_\nu$ is constructed by simply zeroing out certain matrix rows
and columns. In this case, a suitable smoother is obtained by modifying a scalar
Gauß--Seidel method to simply skip all rows with a zero matrix diagonal entry,
and return a zero correction there.

\subsection{Algebraic multigrid iterations}

Geometric multigrid is difficult to apply if the PDE problem is posed on a given grid without
a grid hierarchy.  Luckily, in this situation algebraic multigrid (AMG) iterations can be used
as well~\cite{falgout:2006,stueben:2001}.  In principle, any type of algebraic multigrid
iteration can be used as part of the TNNMG method. The overall TNNMG convergence rates
will depend on the quality of the AMG iteration as an independent solver.  An AMG step that
has merely value as a preconditioner is typically not sufficient because TNNMG expects
suitably scaled inexact solutions to the linear system~\eqref{eq:inexact_newton_problem}.
However, this may be cured by using an additional line search for the linear problem,
i.e., by using single step of an AMG-preconditioned gradient method for~\eqref{eq:inexact_newton_problem}.

If reduction to the truncated correction space $W_\nu$ makes the linear problems semi-definite, then the AMG
step needs to be modified in the same way as the geometric multigrid step described above.
As these modifications only concern the smoothers and additional projection operators, the
techniques used for geometric multigrid  work for AMG, too.

\subsection{Direct solvers and others}

If, for some reason, multigrid solvers are not available, then it is in principle possible
to use any other solver for the correction problems~\eqref{eq:inexact_newton_problem}.
For example, one possible choice is (a number of steps of) a preconditioned CG method,
which will even work out-of-the-box on semidefinite problems.

As a radical choice, it is even possible to use direct solvers, if the restriction to $W_\nu$
is implemented in a way that produces invertible correction matrices.  TNNMG then stops being
a multigrid method, and mutates into something related to nonsmooth Newton methods and
active-set method.  Certain predictor--corrector methods used for small-strain plasticity
problems can be interpreted as variants of the TNNMG method with a direct solver for
the correction problems~\cite{MartinCaddemi1994,sander:2017}.

\appendix
\section{Continuity of minimization}

In this section we present an auxiliary result on the continuity
of minimization operators.
More precisely we show
for $F:\R^n \to \Rinfty$ and a family of subspaces $(V_x)_{x \in \R^n}$ of $\R^n$ that
\begin{align*}
    m : \op{dom} F \to \R,
    \qquad m(x) \colonequals \min_{v\in V_x} F(x+v)
\end{align*}
is continuous if $F$ and the family $(V_x)$ are suitably well behaved.
Only Corollaries~\ref{cor:product_domain_continuity} and~\ref{cor:polyhedral_domain_continuity}
are used in this paper, in Section~\ref{sec:smoothers} on nonlinear smoothers.
We nevertheless state the more general Theorem~\ref{thm:min_continuity} here for
future reference.

\begin{definition}\label{def:stable_K_V}
    Let $\Vcal \colonequals (V_x)_{x\in \R^n}$ be a family of subspaces of $\R^n$.
    We call $K \subset \R^n$ \emph{stable with respect to translations in $\Vcal$}
    if for any sequence $x^\nu \in K$ with $x^\nu \to x \in K$
    we have:
    \begin{enumerate}
        \item
            Any sequence $v^\nu \in V_{x^\nu}$ with $v^\nu \to v$
            satisfies $v \in V_x$.
        \item
            If $v \in V_x$ such that $x+v \in K$, then there is a sequence
            $v^\nu \in V_{x^\nu}$ such that $x^\nu+v^\nu \in K$ and $v^\nu \to v$.
    \end{enumerate}
\end{definition}

The main result shows continuity of $m$ for suitable functions $F$
whose domain $\op{dom} F$ satisfies this stability assumption.

\begin{theorem}\label{thm:min_continuity}
    Let $F:\R^n \to \Rinfty$ such that $F$ is coercive, lower semi-continuous, and continuous
    on its domain.
    Furthermore, let $\Vcal =(V_x)_{x\in \R^n}$ be
    a family of subspaces of $\R^n$ such that
    $K=\op{dom} F$ is stable with respect to translations in $\Vcal$.
    Then $x \mapsto m(x) = \min_{v \in V_x} F(x+v)$ is continuous on $\op{dom} F$.
\end{theorem}

\begin{proof}
    First we note that coercivity and lower semi-continuity
    of $F$ imply that for any $z\in K$ there is exists a minimizer
    $\Acal(z) \colonequals \argmin_{v\in z+V_z} F(z+v) \in K$ (not necessarily unique)
    such that $m(z) = F(\Acal(z))$ is well-defined.

    Now let $x^\nu \in K$ be a sequence with $x^\nu \to x \in K$
    and assume that $F(\Acal(x^\nu)) \not \to F(\Acal(x))$.
    Continuity of $F$ on $K$ and coercivity imply
    \begin{align*}
        -\infty < \min_{z\in \R^n} F(z) \leq F(\Acal(x^\nu)) \leq F(x^\nu) \leq C < \infty.
    \end{align*}
    Hence $F(\Acal(x^\nu))$ is bounded and
    by coercivity $\Acal(x^\nu)$ is also bounded.
    Then there must be a subsequence (w.l.o.g.\ also indexed by $\nu$)
    such that $\Acal(x^\nu) \to \hat{x}$
    and Lemma~\ref{lemma:closed_sublevelset} gives $\hat{x} \in K$.
    Hence we have
    \begin{align}\label{eq:min_subsec_conv}
        F(\Acal(x^\nu)) \to F(\hat{x}) \neq F(\Acal(x))
    \end{align}
    by continuity of $F$ on $K$.
    On the other hand the stability of $K$ with respect to $\Vcal$ together
    with $V_{x^\nu} \ni \Acal(x^\nu) - x^\nu \to \hat{x} - x$ gives $\hat{x} -x \in V_x$
    and thus $F(\hat{x}) > F(\Acal(x))$.

    Finally, the stability of $K$ with respect to $\Vcal$ also implies that there
    is a sequence $v^\nu \in V_{x^\nu}$ such that $x^\nu + v^\nu \in K$
    and $v^\nu \to \Acal(x) - x \in V_x$. Then we have
    \begin{align*}
        F(\Acal(x^\nu)) \leq F(x^\nu + v^\nu) \to F(\Acal(x)) < F(\hat{x}),
    \end{align*}
    which contradicts \eqref{eq:min_subsec_conv}.
\end{proof}

The first condition in Definition~\ref{def:stable_K_V} is trivially
satisfied if $V_x$ is a fixed subspace of $\R^n$.
A simple example where the second condition also holds are
domains with product structure.

\begin{corollary}
    \label{cor:product_domain_continuity}
    Let $F:\R^n \to \Rinfty$ be coercive, lower semi-continuous, and continuous on its domain.
    Furthermore let $K \colonequals \op{dom} F = K_1 \times K_2$ be a product
    of subsets $K_1 \subset \R^{n_1}$ and $K_2 \subset \R^{n_2}$
    and let $V = \R^{n_1} \times \{0\}$.
    Then $K$ is stable with respect to translations in $V$,
    and $m$ is continuous on $K$.
\end{corollary}
\begin{proof}
    To show stability of $K$ with respect to translations in $V$
    let $x^\nu \in K_1 \times K_2$ with $x^\nu \to x \in K_1 \times K_2$
    and $v \in (K_1 \times K_2 - x) \cap V$. Then we have
    $v^\nu = (x_1+v_1-x^\nu_1,0) \in (K-x^\nu) \cap V$ and $v^\nu \to v$.
\end{proof}

Another example of stability in the sense of Definition~\ref{def:stable_K_V}
are convex polyhedral domains. A property similar to the second condition in Definition~\ref{def:stable_K_V}
and related to the stability of tangent cones was shown in
\cite{GraeserSander2014} for this case. We will give a
direct proof for convex polyhedra without using tangent cones.

\begin{corollary}
    \label{cor:polyhedral_domain_continuity}
    Let $F:\R^n \to \Rinfty$ be coercive, lower semi-continuous, and continuous on its domain.
    Furthermore let $K =\op{dom} F$ be a (not necessarily closed)
    convex polyhedron, and $V$ any fixed subspace of $\R^n$.
    Then $K$ is stable with respect to translations in $V$,
    and $m$ is continuous on $K$.
\end{corollary}

\begin{proof}
    First we note that $K$ can be written as
    \begin{equation*}
        K = \big\{ x \in \R^n \st
            \inner{b_i}{x} \lhd_i c_i \text{ for } i=1,\dots,k \big\},
    \end{equation*}
    where  $\lhd_i$ is either ``$<$'' or ``$\leq$'' for each $i$.
    To show stability of $K$ with respect to translations in $V$
    let $x^\nu \in K$ with $x^\nu \to x \in K$, and $v \in V$
    with $x+v \in K$.
    For each $\nu$ and $i$ we define
    \begin{align*}
        \lambda^\nu_i \colonequals
        \begin{cases}
            1 & \text{ if } \inner{b_i}{v} \leq 0,\\
            1 & \text{ if } \inner{b_i}{v} > 0 \text{ and } \inner{b_i}{x^\nu-x} \leq 0,\\
            \max\Big\{0,1-\frac{\inner{b_i}{x^\nu-x}}{\inner{b_i}{v}}\Big\}
                & \text{ if } \inner{b_i}{v} > 0 \text{ and } \inner{b_i}{x^\nu-x} > 0,
        \end{cases}
    \end{align*}
    and $\lambda^\nu = \min_i \lambda^\nu_i \geq 0$.
    Then we have $\inner{x^\nu + \lambda v}{b_i} \lhd_i c_i$ for all $\lambda \in [0,\lambda^\nu_i]$
    and thus $v^\nu = \lambda^\nu v \in (K-x^\nu)\cap V$.
    Furthermore $x^\nu \to x$ implies $\lambda^\nu \to 1$ and hence $v^\nu \to v$.
\end{proof}

%\begin{lemma}
%    Let $K \subset \R^n$ be closed and convex, $F_0 : \R^n \to \R$
%    continuously differentiable, and $F = F_0 + \chi_K$.
%    By $P_{K-x} (-\nabla F_0(x)) \in \R^n$ we denote the projected gradient
%    \begin{align*}
%        P_{K-x} (-\nabla F_0'(x))
%            = \argmin_{v \in K-x} \|v - (-\nabla F_0(x)) \|
%            = P_{K} (x-\nabla F_0'(x)) - x
%    \end{align*}
%    Then $K$ is stable with respect to translations in the
%    family $V=(V_x)_{x\in \R^n}$ with $V_x = \op{span}\{P^K_x F_0'(x)\}$.
%\end{lemma}
%\begin{proof}
%    It is well-known that $P_k$ is continuous.
%    Hence $P_{K-x} (-\nabla F_0'(x))$ depends continuously on $x$.
%    ...
%\end{proof}

\bibliographystyle{abbrv}
\bibliography{graeser_sander_block_separable_tnnmg}

\end{document}